\documentclass[12pt]{amsart}

\usepackage{amsmath,amssymb,amsthm,graphicx,hyperref,pinlabel}
\usepackage{enumitem}
\usepackage{tikz}
\usetikzlibrary{arrows}

\usepackage[margin=1.25in]{geometry}

\numberwithin{equation}{section}
\usepackage{hyperref}
\hypersetup{bookmarksdepth=3}

\theoremstyle{plain}
\newtheorem{theorem}[equation]{Theorem}

\newtheorem{proposition}[equation]{Proposition}
\newtheorem{lemma}[equation]{Lemma}

\newtheorem{conjecture}[equation]{Conjecture}

\newtheorem*{claim*}{Claim}

\newtheorem*{addendum-Thm15}{Addendum to Theorem~\ref{Thm_comparing_RF_weak_form}}

\theoremstyle{remark}
\newtheorem{remark}[equation]{Remark}

\theoremstyle{definition}
\newtheorem{definition}[equation]{Definition}

\newcommand{\acts}{\curvearrowright}

\newcommand{\si}{\sigma}
\newcommand{\wh}{\widehat}
\newcommand{\Ga}{\Gamma}

\newcommand{\C}{{\mathcal C}}

\newcommand{\D}{\partial}

\renewcommand{\H}{\mathbb H}

\newcommand{\M}{{\mathcal M}}
\newcommand{\n}{\mathbb N}
\newcommand{\N}{{\mathcal{N}}}

\newcommand{\R}{\mathbb R}

\renewcommand{\t}{\mathfrak{t}}

\newcommand{\U}{{\mathcal U}}

\newcommand{\Z}{\mathbb Z}

\newcommand{\lb}{\linebreak[1]}

\newcommand{\al}{\alpha}

\newcommand{\be}{\beta}
\newcommand{\ben}{\begin{enumerate}}

\newcommand{\de}{\delta}

\newcommand{\Diff}{\operatorname{Diff}}

\newcommand{\een}{\end{enumerate}}

\newcommand{\Embed}{\operatorname{Embed}}
\newcommand{\eps}{\epsilon}
\newcommand{\face}{\operatorname{Face}}
\renewcommand{\fill}{\operatorname{fill}}
\newcommand{\ga}{\gamma}

\newcommand{\id}{\operatorname{id}}

\newcommand{\isom}{\operatorname{Isom}}

\newcommand{\loc}{\operatorname{loc}}
\newcommand{\lra}{\longrightarrow}
\newcommand{\met}{\operatorname{Met}}

\newcommand{\ol}{\overline}
\newcommand{\om}{\omega}

\newcommand{\partmet}{\operatorname{PartMet}}
\newcommand{\ra}{\rightarrow}
\newcommand{\restr}{\mbox{\Large \(|\)\normalsize}}

\newcommand{\Rm}{\operatorname{Rm}}

\newcommand{\St}{\operatorname{St}}

\newcommand{\ul}{\underline}

\DeclareMathOperator{\Int}{Int}
\DeclareMathOperator{\can}{can}

\DeclareMathOperator{\Ric}{Ric}

\newcommand{\ov}[1]{\overline{#1}}

\def\XXint#1#2#3{{\setbox0=\hbox{$#1{#2#3}{\int}$}
     \vcenter{\hbox{$#2#3$}}\kern-.5\wd0}}

\begin{document}

\author{Richard H. Bamler}
\address{Department of Mathematics, University of California, Berkeley, Berkeley, CA 94720}
\email{rbamler@berkeley.edu}
\author{Bruce Kleiner}
\address{Courant Institute of Mathematical Sciences, New York University,  251 Mercer St., New York, NY 10012}
\email{bkleiner@cims.nyu.edu}
\thanks{The first author was supported by a Sloan Research Fellowship and NSF grant DMS-1611906.
\hspace*{2.38mm} The second author was supported by NSF grants DMS-1405899, DMS-1406394, DMS-1711556, and a Simons Collaboration grant.}

\title[Diffeomorphism groups of $3$-manifolds]{Ricci flow and diffeomorphism groups of $3$-manifolds}

\date{\today}
\maketitle

\begin{abstract}
We complete the proof of the Generalized Smale Conjecture, apart from the case of $RP^3$, and give a new proof of Gabai's theorem for hyperbolic $3$-manifolds.  We use an approach based on Ricci flow through singularities, which applies uniformly to spherical space forms other than $S^3$ and $RP^3$ and hyperbolic manifolds, to prove that the moduli space of metrics of constant sectional curvature is contractible.  As a corollary, for such a $3$-manifold $X$, the inclusion $\isom(X,g)\ra \Diff(X)$ is a homotopy equivalence for any Riemannian metric $g$ of constant sectional curvature.   
\end{abstract}

\tableofcontents

\section{Introduction}
\nocite{cerf}
Let $X$ be a compact connected smooth $3$-manifold.  We let $\Diff(X)$ and $\met(X)$ denote the group of smooth diffeomorphisms of $X$, and the set of Riemannian metrics on $X$, respectively, equipped with their $C^\infty$-topologies.  Our focus in this paper will be on the following conjecture:
\begin{conjecture}[Generalized Smale Conjecture
\cite{smale_mr_review,gabai_smale_conjecture_hyperbolic,rubinstein_et_al}]  \label{conj_gsc} If $g$ is a Riemannian metric of constant sectional curvature $\pm 1$ on $X$, then the inclusion $\isom(X,g)\hookrightarrow\Diff(X)$ is a homotopy equivalence.  
\end{conjecture}

\noindent
Smale's original conjecture was for the case $X=S^3$ \cite{smale_mr_review}.  Cerf proved that the inclusion $\isom(S^3,g)\ra \Diff(S^3)$ induces a bijection on path components \cite{cerf1,cerf2,cerf3,cerf4}, and the full conjecture was proven by Hatcher \cite{hatcher_smale_conjecture}.  Hatcher used a blend of combinatorial and smooth techniques to show that the space of smoothly embedded $2$-spheres in $\R^3$ is contractible.  This  is equivalent to the assertion that $O(4)\simeq \isom(S^3,g) \ra \Diff(S^3)$ is a homotopy equivalence when $g$ has sectional curvature $1$.  Other spherical space forms were studied starting in the late 1970s.  Through the work of a number of authors it was shown that the inclusion $\isom(X)\ra \Diff(X)$ induces a bijection on path components for any spherical space form $X$ \cite{asano,rubinstein_klein_bottles,cappell_shaneson,bonahon,rubinstein_birman,boileau_otal}.  Conjecture~\ref{conj_gsc} was previously known for certain spherical space forms -- those containing geometrically incompressible one-sided Klein bottles (prism and quaternionic manifolds), as well as Lens spaces other than $RP^3$ \cite{ivanov_1,ivanov_2,rubinstein_et_al}.  The conjecture was proven for hyperbolic manifolds by Hatcher and Ivanov in the Haken case \cite{ivanov_haken,hatcher_haken} (extending the earlier work of Waldhausen and Laudenbach \cite{waldhausen,laudenbach}) and by Gabai in general \cite{gabai_smale_conjecture_hyperbolic}.  We recommend \cite[Section 1]{rubinstein_et_al} for a nice discussion of these results and other background on diffeomorphism groups.

In this paper we will use Ricci flow through singularities to prove:
\begin{theorem}
\label{thm_k_equiv_1_contractible}
Let $(X,g)$ be a compact connected Riemannian $3$-manifold of constant sectional curvature $k\in\{\pm 1\}$, other than $S^3$ or $RP^3$, and let  $\met_{K\equiv k}(X)\subset\met(X)$ be the moduli space of Riemannian metrics on $X$ of constant sectional curvature $k$.  Then $\met_{K\equiv k}(X)$ is contractible.
\end{theorem}
\noindent
By a well-known argument (see Lemma~\ref{lem_structure_met_k}), the contractibility of $\met_{K\equiv k}(X)$ is equivalent to the validity of the Generalized Smale Conjecture for $X$. Hence Theorem~\ref{thm_k_equiv_1_contractible} confirms the Generalized Smale Conjecture for several new infinite families of spherical space forms (tetrahedral, octahedral, and icosahedral manifolds), thereby completing the proof of the Generalized Smale Conjecture, apart from the $RP^3$ case.  It also provides a new proof for the other spherical space forms, and for hyperbolic manifolds.  The proof of Theorem~\ref{thm_k_equiv_1_contractible} exploits Ricci flow through singularities as developed in the papers \cite{kleiner_lott_singular_ricci_flows,bamler_kleiner_uniqueness_stability}, and gives a conceptually simple treatment that works uniformly for all manifolds $X$ as in the theorem.  By contrast, the previously known cases of Theorem~\ref{thm_k_equiv_1_contractible} were established using traditional tools from $3$-manifold topology.   They rely on the presence of certain types of distinguished surfaces:  geometrically incompressible Klein bottles or surfaces obtained from sweepouts in the spherical space form cases \cite{ivanov_1,ivanov_2,rubinstein_et_al}, or canonical solid tori arising from the Gabai's insulator techniques in the hyperbolic case \cite{gabai_smale_conjecture_hyperbolic}.  

The method used in this paper breaks down for $S^3$ and $RP^3$ due to the geometric structure of the ``thin'' part of a Ricci flow through singularities.  We will treat these cases in a separate paper \cite{bamler_kleiner_in_prep} using a more involved approach (still based on Ricci flow). Ricci flow also gives a strategy for analyzing diffeomorphism groups of some other families of $3$-manifolds.  We will discuss this in a forthcoming paper.

We remark that it has been a longstanding question whether it is possible to use techniques from geometric analysis to analyze diffeomorphism groups in dimension $3$.   There are variety of natural variational approaches to studying the space of $2$-spheres in $\R^3$ (or $S^3$) that break down due to the absence of a Palais-Smale condition, because there are too many critical points, or because the natural gradient flow does not respect embeddedness; analogous issues plague other strategies based more directly on diffeomorphisms. Theorem~\ref{thm_k_equiv_1_contractible} is the first instance where techniques from geometric analysis have been successfully applied to the study of diffeomorphism groups of $3$-manifolds.  This success depends crucially on the recent results establishing existence and uniqueness of Ricci flow through singularities for arbitrary initial conditions \cite{kleiner_lott_singular_ricci_flows,bamler_kleiner_uniqueness_stability}.

\bigskip

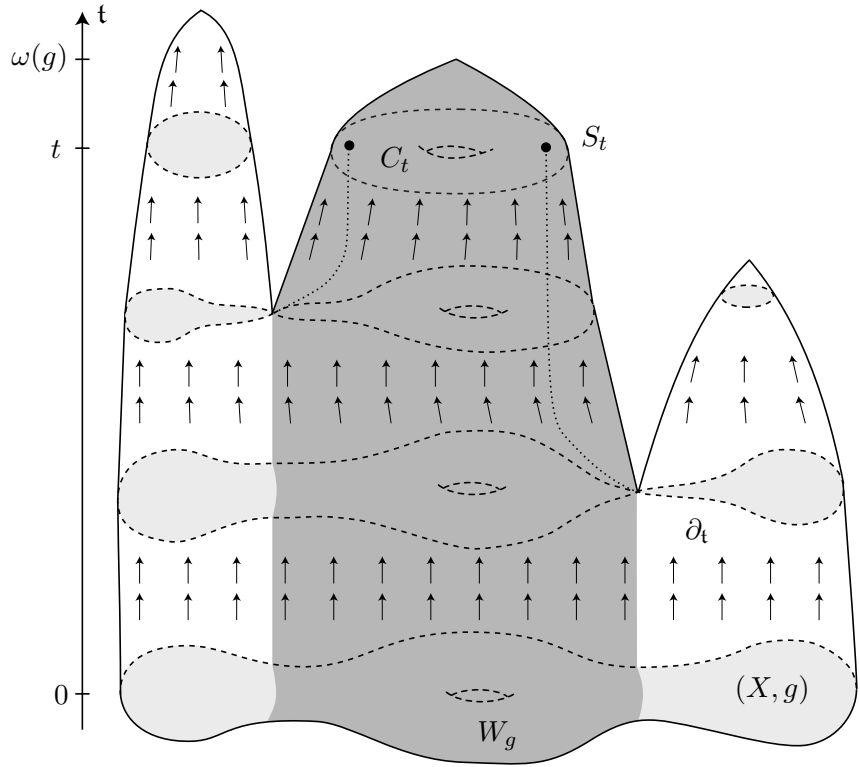
\begin{figure}
\labellist
\small\hair 2pt
\pinlabel $0$ at -20 137
\pinlabel $\omega(g)$ at -60 1320
\pinlabel $\t$ at 55 1400
\pinlabel $t$ at -30 1150
\pinlabel {$(X,g)$} at 1300 145
\pinlabel $W_g$ at 790 60
\pinlabel $C_t$ at 600 1130
\pinlabel $S_t$ at 970 1170
\pinlabel $\partial_{\t}$ at 1160 440
\endlabellist
\centering
\includegraphics[width=105mm]{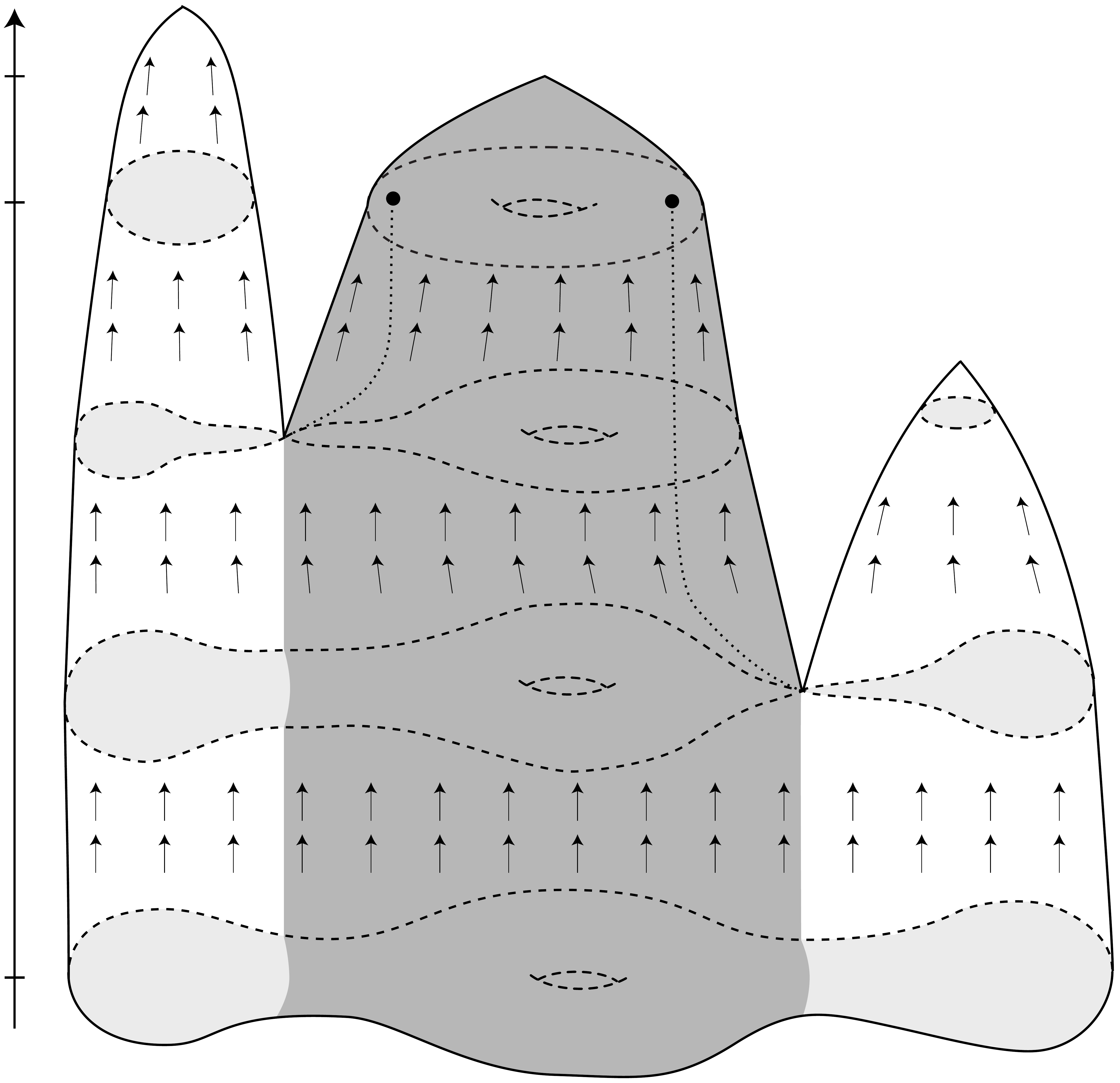}
\caption{A Ricci flow spacetime $\M$ (aka singular Ricci flow) with initial data $(X, g)$.
The time function $\t$ is expressed as height function and the arrows indicate the time vector field $\partial_\t$.
The dashed level sets indicate time slices (the second and third time slices are singular). 
For all $t < \omega(g)$ the time-$t$-slice contains exactly one component $C_t$ that is not diffeomorphic to a sphere.
The flow of $\partial_\t$ restricted to $W$ is defined on the time interval $[0, \omega(g))$.
The union of these trajectories is shaded in dark gray.
For each $t$ the trajectory of $\partial_\t$ starting from each but a finite set of points $S_t \subset C_t$ intersects the time-$0$-slice.
These points are drawn as dark dots and their trajectories, which cease to exist at a singular time, as dotted curves.
\label{fig_spacetime}}
\end{figure}
\subsection*{Informal sketch of proof} Let $X$ be as in the statement of Theorem~\ref{thm_k_equiv_1_contractible}.  To simplify notation we will focus on the case in which $X$ is a spherical space form; at the end we comment on the modifications needed for the hyperbolic case.  Thus our aim is to show that $\met_{K\equiv 1}(X)$ is contractible, which reduces to showing that all of its homotopy groups are trivial.  

Let $g_X\in \met_{K\equiv 1}(X)$ be a reference metric. It is a classical fact that any two metrics in $\met_{K\equiv 1}(X)$ are isometric.

Before proceeding, we first recall some the properties of Ricci flow through singularities, as established in \cite{kleiner_lott_singular_ricci_flows,bamler_kleiner_uniqueness_stability}. We keep the discussion informal, and refer the reader to Section~\ref{sec_prelim} for precise definitions and references.  

For every $g\in\met(X)$, there exists a singular Ricci flow  with initial data $(X,g)$.  This is a Ricci flow spacetime, i.e. a $4$-manifold $\M$ equipped with a time function $\t$ and time vector field $\D_\t$, as well as a Riemannian metric along time slices that satisfies the Ricci flow equation.  Locally, the spacetime looks like an a piece of a Ricci flow defined for a short time in some open subset of $U\subset\R^3$, and the trajectories of the time vector field $\D_\t$ correspond to the spacetime tracks of points that are motionless.  The time-$t$ slice $\M_t$ of $\M$ is the result of evolving the metric $(X,g)$ under Ricci flow for a duration $t$.  For small $t\geq 0$ this corresponds to the usual Ricci flow, but singularities may develop subsequently, which can result in noncompact time slices.  Although the structure of the spacetime may be rather complicated, it still has good topological and geometric properties:
\begin{enumerate}[label=(\alph*)]
\item For every $t$, at most one connected component of the time slice $\M_t$ is diffeomorphic to $X$ with possibly finitely many punctures,  while the remaining components are topologically trivial -- copies of $S^3$ with possibly finitely many punctures.  
\item There is a $T<\infty$, depending only on bounds on the geometry of $g$, such that $\M_t=\emptyset$ for $t\geq T$.  We establish this fact using the extinction results from \cite{colding_minicozzi_extinction,perelman_extinction}, and the fact that $X$ is prime and not aspherical.
\item Let $\om(g)\in (0,T]$ be the supremum of the times $t\in [0,\infty)$ for which $\M_t$ has a topologically nontrivial component.  As $t\ra \om(g)$, the time slice $\M_t$ has a unique component $C_t$ diffeomorphic to $X$, and $C_t$ becomes asymptotically round as $t\ra\om(g)$, i.e. the family of Riemannian manifolds $(C_t)_{t<\om(g)}$ converges, modulo rescaling, to $(X,g_X)$.
\een 
We remark that assertion (c) is based on rigidity properties of $\kappa$-solutions (the geometric models for the large curvature part of $\M$), and it makes use of the assumption that $X$ is not diffeomorphic to $S^3$ or $RP^3$ in order to exclude more complicated geometric behavior as $t\ra \om(g)$. 

Now consider a time $t<\om(g)$ close to $\om(g)$.  By \cite{kleiner_lott_singular_ricci_flows}, there is a finite subset $S_t\subset C_t$ such that the entire complement $C_t\setminus S_t$ lies in the domain of the time-$(-t)$-flow of the time vector field.  In other words, every point in $C_t\setminus S_t$ lies on a trajectory of the time vector field starting from the time-$0$ slice $\M_0=(X,g)$.    Hence we may pushforward the metric on $C_t\setminus S_t$ under the flow, to obtain a Riemannian metric $\check g_t$ on open subset $W_g(t)$ of the time-$0$ slice $\M_0=X$.  We prove that $\check g_t$ converges, modulo rescaling, to a metric $\check g\in \met_{K\equiv 1}(W_g)$, where $W_g\subset X$ is an open subset, and the Riemannian manifold $(W_g,\check g)$ is isometric $(X\setminus S_g,g_X)$ for some finite set $S_g$.  

To summarize, using singular Ricci flow we have taken an arbitrary Riemannian metric $g\in \met(X)$, and produced a Riemannian metric of constant sectional curvature $1$, albeit one defined only on some open subset $W_g\subset X$.  We point out that although $W_g$ might in principle be rather wild, it still contains the interesting topology of $X$ because it is diffeomorphic to $X\setminus S_g$ for some finite set $S_g\subset X$.

Note that if $g$ has constant sectional curvature $1$, then $\M$ corresponds to an ordinary Ricci flow and we have $C_t=\M_t$, $W_g(t)=X$ for all $t<\om(g)$, and $(W_g,\check g)=(X,g)$.

Since the singular Ricci flow $\M$ is unique up to isometry \cite{bamler_kleiner_uniqueness_stability}, the partially defined constant curvature metric $(W_g,\check g)$ is canonically attached to $g$.  Furthermore, using the stability theorem \cite{bamler_kleiner_uniqueness_stability} and \cite{knopf_et_al}, we  show that $(W_g,\check g)$ depends continuously on $g$, in an appropriate sense.  

We now return to the task of showing that the homotopy groups of $\met_{K\equiv 1}(X)$ are trivial.

Pick $m\geq 0$, and consider a (continuous) map $h:S^m\ra \met_{K\equiv 1}(X)$.  Our goal is to extend $h$ to a map $\hat h:D^{m+1}\ra \met_{K\equiv 1}(X)$.  Since $\met(X)$ is contractible, there is an extension $g:D^{m+1}\ra \met(X)$   of the composition $S^m\stackrel{h}{\ra}\met_{K\equiv 1}(X)\hookrightarrow\met(X)$.  

For every $p\in D^{m+1}$, let $(W_{g(p)},\check g(p))$ be the partially defined metric described in the preceding paragraphs.  
Note that $(W_{g(p)},\check g(p))=(X,g(p))$ when $p\in S^m$.
 To complete the proof, we show that after shrinking $W_{g(p)}$ slightly, one can extend $\check g(p)$ to a metric $\wh h(p)$ with sectional curvature $1$ defined on all of $X$, where $\wh h(p)$ depends continuously on $p$.  

We now give an indication of the extension process.  Pick $p\in D^{m+1}$.  Since $W_{g(p)}$ is diffeomorphic to $X\setminus S_{g(p)}$ and $X$ is irreducible, there is a compact domain with boundary $Z_p\subset W_p$, such that the closure $\ol{X\setminus Z_p}$ is a finite disjoint collection of closed $3$-disks.  We would like to extend the restriction $\check g(p)\restr Z_p$ across each of the $3$-disk components of $\ol{X\setminus Z_p}$ to obtain $\wh h(p)\in\met_{K\equiv 1}(X)$.   Pick one such $3$-disk $D$.  It is not hard to see that the extension problem is equivalent to an extension problem for embeddings: for a suitable open neighborhood $U$ of the boundary $\D D$ in $D$, one is given a smooth embedding of $U$ into the round $3$-sphere, and one has to extend this to an embedding $D\ra S^3$. Hatcher's theorem \cite{hatcher_smale_conjecture} implies that this problem has a contractible solution set.   To handle the full extension problem, we take a suitable fine triangulation of $D^{m+1}$, and carry out a parametrized analog of this extension procedure,  by induction over the skeleta.  
This is similar in spirit to an argument using obstruction theory, where the obstruction group is trivial.

We now discuss the hyperbolic case.   Suppose $X$ is a hyperbolic manifold, and pick a hyperbolic metric $g_X\in\met_{K\equiv -1}(X)$; for simplicity we assume here that $X$ is orientable.  

Any $g\in\met(X)$ can be evolved into a singular Ricci flow $\M$ as before.  Its properties are similar to those in the spherical space form case, except that assertions (b) and (c) have to be modified: for every $t\in [0,\infty)$ there is a unique component $C_t$ of $\M_t$ that is diffeomorphic to a punctured copy of $X$, and as $t\ra \infty$ the family of Riemannian manifolds $(C_t)_{t<\infty}$ converges, modulo rescaling, to $(X,g_X)$.  Proceeding as before we use this to construct a canonical partially defined metric $(W_g,\check g)$ with sectional curvature $-1$, where $W_g\subset X$, and $(W_g,\check g)$ is isometric to $(X\setminus S_g,g_X)$ for some finite subset $S_g\subset X$.  The rest of the proof is essentially the same as for spherical space forms.

\bigskip
\begin{remark}
We point out that one may use singular Ricci flow to show that any two metrics $g(0),g(1)\in\met_{K\equiv 1}(X)$ are isometric, without appealing to Reidemeister or Whitehead torsion \cite{milnor_whitehead_torsion}.  (Of course the Ricci flow proof is vastly more complicated than proofs using torsion, since it invokes Perelman's work as well as \cite{bamler_kleiner_uniqueness_stability}.)  The idea is as follows.  Let $g:[0,1]\ra \met(X)$ be a path from $g(0)$ to $g(1)$.  For every $p\in [0,1]$, we let $\M^p$ be the singular Ricci flow with $\M^p_0=(X,g_t)$.  As explained in the sketch above, the spacetime $\M^p$ contains a family $\{C^p_t\}$ of time slices that become asymptotically round as $t\ra \om(g(p))$.  This may be used to construct a family $\{(\ov C^p,\ov g(p))\}_{p\in [0,1]}$ of compact Riemannian manifolds with constant sectional curvature $1$  which interpolates between $(X,g(0))$ and $(X,g(1))$ and which varies continuously in the smooth topology on Riemannian manifolds.  Therefore the set of isometry classes of such metrics,  equipped with the smooth topology on Riemannian manifolds, is connected.  On the other hand, one knows that the space of isometry classes is finite: this follows from the isometric classification of spherical space forms, or alternatively, from a simple general argument based on the finiteness of the set of irreducible representations of a finite group.  Hence it contains a single point.

The same remark also applies to the hyperbolic case -- using singular Ricci flow one can give a new proof of Mostow rigidity assuming only local rigidity of hyperbolic metrics (in the appropriate form).  However, to carry this out one would have to modify the existing large-time analysis slightly so that it only invokes local rigidity rather than Mostow-Prasad rigidity.
\end{remark}

\section{Preliminaries}
\label{sec_prelim}

\subsection{Spaces of maps and metrics}
If $M$, $N$ are smooth manifolds with boundary, we let  $\Embed(M,N)$  denote the set of smooth embeddings $M\ra N$ equipped with the $C^\infty_{\loc}$-topology.

If $M$ is a smooth manifold, we let $\met(M)$ denote the set of smooth Riemannian metrics on $M$ equipped with the $C^\infty_{\loc}$-topology.  For $k\in \R$, we let $\met_{K\equiv k}(M)$ be the subspace of metrics with constant sectional curvature $k$.

We will need the following consequence of the Smale Conjecture \cite{hatcher_smale_conjecture}.

\begin{lemma}
\label{lem_embedding_restriction_fiber_bundle}
Let $\Embed_+(D^3,S^3)\subset \Embed(D^3,S^3)$ be the subset of orientation-preserving embeddings, and let 
$$
\pi:\Embed_+(D^3,S^3)\lra\Embed(S^2,S^3)
$$
be the map induced by restriction.  Then:
\begin{enumerate}[label=(\alph*)]
\item 
$\pi:\Embed_+(D^3,S^3)\lra\Embed(S^2,S^3)$ is a fiber bundle with contractible fiber.
\item Let $m\geq 0$.  Suppose $\phi_{m+1}:D^{m+1}\ra \Embed(S^2,S^3)$ is a continuous map and $\wh\phi_m:S^m\ra \Embed_+(D^3,S^3)$ is a lift of $\phi_{m+1}\restr S^m$, i.e. $\pi\circ \wh\phi_m=\phi_{m+1}\restr S^m$.  Then there is an extension $\wh\phi_{m+1}:D^{m+1}\ra \Embed_+(D^3,S^3)$ of $\wh\phi_m$ that is a lift of $\phi_{m+1}$.
\een
\end{lemma}
\begin{proof}
The fact that $\pi$ is a fiber bundle is a standard consequence of isotopy extension, which we briefly recall (see \cite{rubinstein_et_al}).  

Given $f_0\in \Embed(S^2,S^3)$ there is an open neighborhood $N(f_0)$ of $f_0$ and a continuous map 
$$
\Phi:N(f_0) \ra \Diff(S^3)
$$
such that for all $f\in N(f_0)$ we have $f=\Phi(f)\circ f_0$.  The map $\Phi$ may be obtained by constructing a locally defined isotopy near $f_0(S^2)$ using normal exponential maps, and then gluing this to the identity map with a partition of unity.  Letting $F:=\pi^{-1}(f_0)$, we obtain a bundle chart $N(f_0) \times F\ra \pi^{-1}(N(f_0))$ for $\pi$ by sending $(f,\phi)$ to $\Phi(f)\circ \phi$.

By \cite[p.604]{hatcher_smale_conjecture}, the subset
$$
\Diff(D^3rel\D D^3):=\left\{\al\in\Diff(D^3) \mid \al\restr_{S^2}=\id_{S^2}\right\}
$$
is contractible.  If $\phi_0\in F$, then we obtain a homeomorphism 
$$
\Diff(D^3rel\D D^3):=\left\{\al\in\Diff(D^3)\mid \al\restr_{S^2}=\id_{S^2}\right\}\lra F
$$
by sending $\al$ to $\phi_0\circ \al$.  Hence $F$ is also contractible.   Thus assertion (a) holds.

Since $\pi$ is a fiber bundle with contractible fiber, any map can be lifted relative to its boundary.  Hence assertion (a) holds.
\end{proof}

The following is well-known:
\begin{lemma}
\label{lem_structure_met_k}
Let $X$ be a compact connected $3$-manifold, and $g_X\in\met_{K\equiv k}(X)$ for $k\in \{\pm 1\}$.  Then:
\begin{itemize}
\item There is a fibration $\Diff(X)\ra \met_{K\equiv k}(X)$ with fiber homeomorphic to $\isom(X, \linebreak[1] g_X)$.
\item $\Diff(X)$ and $\met_{K\equiv k}(X)$ are homotopy equivalent to CW complexes.
\item $\met_{K\equiv k}(X)$ is contractible if and only if the inclusion $\isom(X,g_X)\ra \Diff(X)$ is a homotopy equivalence.
\end{itemize}
\end{lemma}
\begin{proof}[Sketch of proof]
The metric $g_X$ is unique up to isometry; this follows from Mostow rigidity when $k=-1$, and by the isometric and smooth classification of spherical space forms when $k=1$.  Therefore the action $\Diff(X)\acts\met_{K\equiv k}(X)$ by pushforward is transitive, with stabilizer $\isom(X,g_X)$.  The space $\met_{K\equiv k}(X)$ is then homeomorphic to the orbit space $\Diff(X)/\isom(X,g_X)$, and we have a fibration $\Diff(X)\ra\met_{K\equiv k}(X)$ with fiber homeomorphic to $\isom(X,g_X)$.  The diffeomorphism group $\Diff(X)$ is a Frechet manifold that is locally diffeomorphic to the space $\Ga_{C^\infty}(TX)$ of $C^\infty$ vector fields on $X$.  Using the orbit space representation, one gets that $\met_{K\equiv k}(X)$ is a separable Frechet manifold modelled on a finite codimension closed subspace of $\Ga_{C^\infty}(TX)$.  Hence both spaces have the homotopy type of CW complexes.   Finally, using the exact homotopy sequence of the fibration $\Diff(X)\ra\met_{K\equiv k}(X)$, we get  
\begin{align*}
&\qquad\quad\met_{K\equiv k}(X)\text{ is  contractible}\\
&\iff \met_{K\equiv k}(X)\text{ is weakly contractible}\\
&\iff\text{ The inclusion $\isom(X,g_X)\ra\Diff(X)$ is a weak homotopy equivalence}\\
&\iff\text{ The inclusion $\isom(X,g_X)\ra\Diff(X)$ is a homotopy equivalence.}
\end{align*}
\end{proof}

We will also work with the collection of Riemannian metrics defined on different subsets of a given manifold.

\begin{definition}[Topology on partially defined metrics]
\label{def_topology_partially_defined_metrics}
Let $M$ be a smooth manifold, and let $\partmet(M)$ be the set of partially defined Riemannian metrics on $M$, i.e. the set of pairs $(U,h)$ where $U\subset M$ is open and $h$ is a smooth Riemannian metric on $U$.    We topologize $\partmet(M)$ as follows.  For every $(U_0,h_0)\in\partmet(M)$,  $K\subset U_0$ compact, $k<\infty$, and $\eps>0$,  we let
\begin{align*}
\U(U_0,h_0,K,k,\eps):=
\{ (U,h)\in &\partmet(X) \mid K\subset U,\; \|(\nabla_{h_0}^j(h-h_0))(x)\|_{h_0}<\eps\\
&\text{ for all }x\in K, \;j\leq k \}\,.
\end{align*}
The collection of all such subsets $\U(U_0,h_0,K,k,\eps)$  is a basis for the topology on $\partmet(X)$. 
\end{definition}

\bigskip
Note that if $Z$ is a metric space, then in order to verify that a map $Z\ni z\mapsto (U(z),g(z))\in \partmet(M)$ is continuous, it suffices to show that if $z_j\ra z_\infty\in Z$, then for every compact subset $K\subset U(z_\infty)$ we have $K\subset U(z_j)$ for large $j$, and $\nabla^k_{g(z_\infty)}(g(z_j)-g(z_\infty))\ra 0$ uniformly on $K$.

\begin{remark}
The topology on $\partmet(M)$ has the somewhat alarming property of being non-Hausdorff.  This is due to the fact that it formalizes the lower semicontinuous dependence of the open set $U$.  It may be compared to the non-Hausdorff topology on $\R$ generated by the set of open rays $\{(a,\infty)\}_{a\in\R}$, which may be used to characterize lower semicontinuous real-valued functions $X\ra \R$.
\end{remark}


\subsection{Closeness and convergence of Riemannian manifolds}
We now recall notions of closeness for Riemannian manifolds.
\begin{definition}[Geometric closeness] \label{def_geometric_closeness_time_slice}
We say that a pointed Riemannian manifold $(M, g, x)$ is \textbf{$\eps$-close} to another pointed Riemannian manifold $(\ov{M}, \ov{g}, \ov{x})$ \textbf{at scale $\lambda > 0$} if there is a diffeomorphism onto its image
\[ \psi : B^{\ov{M}} (\ov{x}, \eps^{-1} ) \longrightarrow M \]
such that $\psi (\ov{x}) = x$ and
\[ \big\Vert \lambda^{-2} \psi^* g - \ov{g} \big\Vert_{C^{[\eps^{-1}]}(B^{\ov{M}} (\ov{x}, \eps^{-1} ))} < \eps. \]
Here the $C^{[\eps^{-1}]}$-norm  of a tensor $h$ is defined to be the sum of the $C^0$-norms of the tensors $h$, $\nabla^{\ov{g}} h$, $\nabla^{\ov{g},2} h$, \ldots, $\nabla^{\ov{g}, [\eps^{-1}]} h$ with respect to the metric $\ov{g}$.  We say that 
$(M, g, x)$ is \textbf{$\eps$-close} to $(\ov{M}, \ov{g}, \ov{x})$ if it is $\eps$-close at scale $1$.  We have analogous notions for (unpointed) Riemannian manifolds: \textbf{$(M,g)$ is $\eps$-close to $(\ov M,\ov g)$} if there is a diffeomorphism $\psi:\ov M\ra  M$ such that
\[ \big\Vert  \psi^* g - \ov{g} \big\Vert_{C^{[\eps^{-1}]}(\ov M)} < \eps. \]
\end{definition}
The notion of closeness provides a notion of convergence of sequences (or families) of Riemannian manifolds, in the usual way.

\bigskip\bigskip

\subsection{Ricci flow spacetimes}
We now recall the properties of singular Ricci flows that will be essential in this paper.  We refer the reader to \cite{kleiner_lott_singular_ricci_flows,bamler_kleiner_uniqueness_stability} for more details.

\bigskip
\begin{definition}[Ricci flow spacetimes] \label{def_RF_spacetime}
A {\bf Ricci flow spacetime (starting at time $a\in \R$)}  is a tuple $(\M, \lb \mathfrak{t}, \lb \partial_{\mathfrak{t}}, \lb g)$ with the following properties:
\begin{enumerate}[label=(\arabic*)]
\item $\M$ is a smooth $4$-manifold with (smooth) boundary $\partial \M$.
\item $\mathfrak{t} : \M \to [a, \infty)$ is a smooth function without critical points (called {\bf time function}).
For any $t \geq a$ we denote by $\M_t:=\mathfrak{t}^{-1} (t) \subset \M$ the {\bf time-$t$-slice} of $\M$.
\item We have $\M_a = \mathfrak{t}^{-1} (a) = \partial \M$, i.e. the initial time-slice is equal to the boundary of $\M$.
\item $\partial_{\mathfrak{t}}$ is a smooth vector field (the {\bf time vector field}), which satisfies $\partial_{\mathfrak{t}} \mathfrak{t} \equiv 1$.
\item $g$ is a smooth inner product on the spatial subbundle $\ker (d \mathfrak{t} ) \subset T \M$.
For any $t \geq a$ we denote by $g_t$ the restriction of $g$ to the time-$t$-slice $\M_t$ (note that $g_t$ is a Riemannian metric on $\M_t$).
\item $g$ satisfies the Ricci flow equation: $\mathcal{L}_{\partial_\mathfrak{t}} g = - 2 \Ric (g)$.
Here $\Ric (g)$ denotes the symmetric $(0,2)$-tensor on $\ker (d \mathfrak{t} )$ that restricts to the Ricci tensor of $(\M_t, g_t)$ for all $t \geq a$.
\end{enumerate}
For any interval $I \subset [a,\infty)$ we also write $\M_{I} = \mathfrak{t}^{-1} (I)$ and call this subset the {\bf time-slab} of $\M$ over the time interval $I$.  
Curvature quantities on $\M$, such as the Riemannian curvature tensor $\Rm$, the Ricci curvature $\Ric$, or the scalar curvature $R$ will refer to the corresponding quantities with respect to the metric $g_t$ on each time-slice.
Tensorial quantities will be embedded using the splitting $T\M = \ker (d\mathfrak{t} ) \oplus \langle \partial_{\mathfrak{t}} \rangle$.

Unless otherwise specified, we will implicitly take $a=0$.  
When there is no chance of confusion, we will usually abbreviate the tuple $(\M, \mathfrak{t}, \partial_{\mathfrak{t}}, g)$ by $\M$.
\end{definition}

\begin{definition}[Survival] \label{def_points_in_RF_spacetimes}
Let $(\M, \mathfrak{t}, \partial_{\mathfrak{t}}, g)$ be a Ricci flow spacetime and $x \in \M$ be a point.
Set $t := \mathfrak{t} (x)$.
Consider the maximal trajectory $\gamma_x : I \to \M$, $I \subset [0, \infty)$ of the time-vector field $\partial_{\mathfrak{t}}$ such that $\gamma_x (t) = x$.
Note that then $\mathfrak{t} (\gamma_x(t')) = t'$ for all $t' \in I$.
For any $t' \in I$ we say that $x$ \textbf{survives until time $t'$}, and we write 
\[ x(t') := \gamma_x (t'). \]

Similarly, if $X \subset \M_t$ is a subset in the time-$t$ time-slice, then we say that $X$ \textbf{survives until time $t'$} if this is true for every $x \in X$ and we set $X(t') := \{ x(t') \;\; : \;\; x \in X \}$.
\end{definition}

A {\bf product Ricci flow spacetime} is a Ricci flow spacetime associated with an ordinary Ricci flow $(g(t))_{t\in [a,T)}$ on a manifold $M$, i.e. it is of the form $(M\times [a,T),\t,\D_\t,g)$, where  $\t=\pi_{[a,T)}$ is projection onto the interval factor $[a,T)$, $\D_\t$ corresponds to the vector field coming from $[a,T)$, and $g_t=\pi_M^*g(t)$ where $\pi_M:M\times[a,T)\ra M$ is the canonical projection. 

\begin{definition}[Product domain]
\label{def_product_domain}
Let $(\M, \mathfrak{t}, \partial_{\mathfrak{t}}, g)$ be a Ricci flow spacetime and let $X \subset \M$ be a subset.
We call $X$ a \emph{product domain} if there is an interval $I \subset [0, \infty)$ such that for any $t \in I$ any point $x \in X$ survives until time $t$ and $x(t) \in X$.
\end{definition}

\begin{definition}[Completeness of Ricci flow spacetimes] \label{def_completeness}
We say that a Ricci flow spacetime $(\M,\mathfrak{t}, \partial_{\mathfrak{t}}, g)$ is {\bf $0$-complete} if, whenever $\ga:[0,s_0)\ra \M$ is either an integral curve of $\pm\D_\t$ or a unit speed curve in some time slice, and $\sup_{s\in [0,s_0)}|\Rm|(\ga(s))<\infty$, then $\lim_{s\ra s_0}\ga(s)$ exists.
\end{definition}

\bigskip
The next lemma states that maximal product domains in $0$-complete Ricci flow spacetimes correspond to ordinary Ricci flows, provided their time slices are compact manifolds.
\begin{lemma}
\label{lem_maximal_product_region}
Let $(\M,\t,\D_\t,g)$ be a $0$-complete Ricci flow spacetime, $t_0\geq 0$, and $C\subset\M_{t_0}$ be a compact $3$-dimensional submanifold without boundary (i.e. a finite union of compact components).  Let $\C\subset\M$ be the maximal product domain in $\M_{\geq t_0}$ with initial time slice $C$, and $(h(t))_{t\in [t_0,T)}$ be the maximal Ricci flow  on $C$ whose initial metric $h(t_0)$ is equal to the restriction of the time slice metric $g_{t_0}$  to $C\subset\M_{t_0}$.  Then $\C$ is isometric to the product Ricci flow spacetime associated with $(h(t))_{t\in [t_0,T)}$.
\end{lemma}
\begin{proof}
Let $I\subset [t_0,\infty)$ be the time interval on which $\C$ is defined.  The compactness of $C$ implies that $I=[t_0,t_+)$ for some $t_+\in(t_0,\infty]$.  For $t\in [t_0,t_+)$, let $g_{t,t_0}$ be the Riemannian metric on $C$ obtained by pushing forward $g\restr \C_t$ under the time-$(t_0-t)$ flow of the time vector field $\D_\t$.  It follows from the definition of Ricci flow spacetimes that the family of metrics $(g_{t,t_0})_{t\in[t_0,t_+)}$ defines a Ricci flow.  By the uniqueness of Ricci flow, we therefore have $g_{t,t_0}=h(t)$ for $t<\min(T,t_+)$.  It follows that $t_+\leq T$.  Suppose $t_+<T$.  Choose $x\in C$.  Since $t_+<T$, the curvature $|{\Rm}|(x(t))$ remains uniformly bounded for $t\in [t_0,t_+)$, and hence by $0$-completeness, $x$ survives until time $t_+$, and hence to some $t'>t_+$.  By continuity of the flow of $\D_\t$, an open neighborhood of $x$ in $C$ survives until some time $t_x>t_+$.  By compactness $C$ survives until some time $t'>t_+$, which is a contradiction.
\end{proof}

\bigskip

\subsection{Singular Ricci flows} 

To define singular Ricci flows, we require the definition of a $\kappa$-solution.

\begin{definition}[$\kappa$-solution] \label{def_kappa_solution}
An ancient Ricci flow $(M, (g(t))_{t \in (-\infty, 0]} )$ on a $3$-di\-men\-sio\-nal manifold $M$ is called a \textbf{(3-dimensional) $\kappa$-solution}, for $\kappa > 0$, if the following holds:
\begin{enumerate}[label=(\arabic*)]
\item $(M, g(t))$ is complete for all $t \in (- \infty, 0]$,
\item $|{\Rm}|$ is bounded on $M \times I$ for all compact $I \subset ( - \infty, 0]$,
\item $\sec_{g(t)} \geq 0$ on $M$ for all $t \in (- \infty, 0]$,
\item $R > 0$ on $M \times (- \infty, 0]$,
\item $(M, g(t))$ is $\kappa$-noncollapsed at all scales for all $t \in (- \infty, 0]$

(This means that for any $(x,t) \in M \times (- \infty, 0]$ and any $r > 0$ if $|{\Rm}| \leq r^{-2}$ on the time-$t$ ball $B(x,t,r)$, then we have $|B(x,t,r)| \geq \kappa r^n$ for its volume.)
\end{enumerate}
\end{definition}

We can now define the canonical neighborhood assumption.  This characterizes the local geometry of a Ricci flow spacetime by the geometry of $\kappa$-solution using the notion of pointed closeness from Definition~\ref{def_geometric_closeness_time_slice}.  The main statement of this assumption is that regions of small scale (i.e. high curvature) are geometrically close to regions of $\kappa$-solutions.

\begin{definition}[Canonical neighborhood assumption] \label{def_canonical_nbhd_asspt}
Let $(M, g)$ be a (possibly incomplete) Riemannian manifold.
We say that $(M, g)$ satisfies the {\bf $\eps$-canonical neighborhood assumption} at some point $x$ if there is a $\kappa > 0$, a $\kappa$-solution $(\overline{M}, \linebreak[1] (\ov{g}(t))_{t \in (- \infty, 0]})$ and a point $\ov{x} \in \ov{M}$ such that $|{\Rm}| (\overline{x}, 0) = 1$ and such that $(M, g, x)$ is $\eps$-close to $(\ov{M}, \ov{g}(0), \ov{x})$ at some (unspecified) scale $\lambda > 0$.

For  $r>0$, we say that a subset $X$ of a Ricci flow spacetime $(\M,\t,\D_\t,g)$ satisfies the \textbf{$\eps$-canonical neighborhood assumption at scales below $r$}   if the $\eps$-canonical neighborhood assumption holds at all $x \in X$ with $|{\Rm}|(x)>r^{-2}$.
\end{definition}

\begin{definition}[Singular Ricci flow] 
  If $\eps>0$ and $r:[0,\infty)\ra (0,\infty)$ is a nonincreasing function, then an {\bf $(\eps,r)$-singular Ricci flow} is an orientable Ricci flow spacetime $(\M,\t,\D_\t,g)$ such that:
\begin{itemize}
\item  The initial time slice $\M_0$ is compact.
\item $\M$ is $0$-complete.
\item $\M_{[0,t]}$ satisfies the $\eps$-canonical neighborhood assumption at scales $<r(t)$.
\end{itemize}
A {\bf singular Ricci flow} is a Ricci flow spacetime that is an $(\eps,r)$-singular Ricci flow for some $\eps$, $r$.
\end{definition}

We remark that our notion of singular Ricci flow here is equivalent to the one in \cite{bamler_kleiner_uniqueness_stability}, which is weaker than the one in
\cite{kleiner_lott_singular_ricci_flows}. The existence theorem in \cite{kleiner_lott_singular_ricci_flows} yields singular Ricci flows satisfying the stronger condition. 
\begin{theorem}[Existence and uniqueness of singular Ricci flow]
\label{thm_existence_uniqueness_singular_ricci_flow}

\mbox{}
\begin{itemize}
\item {\rm (Existence \cite{kleiner_lott_singular_ricci_flows})}
For every compact orientable Riemannian $3$-manifold $(M, \linebreak[1] g)$ there is a singular Ricci flow $\M$ with $\M_0$ isometric to $(M,g)$.  Moreover, for every $\eps>0$ there is an $r:[0,\infty)\ra (0,\infty)$ such that $\M$ is an $(\eps,r)$-singular Ricci flow, where $r$ depends only on $\eps$ and  an upper bound on $|\Rm|$ and a lower bound on the injectivity radius of $M$.
\item {\rm (Uniqueness \cite{bamler_kleiner_uniqueness_stability})}  There is a universal constant $\eps_{\can}>0$ such that if $\M^1$, $\M^2$ are $(\eps_{\can},r)$-singular Ricci flows for some $r:[0,\infty)\ra(0,\infty)$, then any isometry $\phi:\M^1_0\ra\M^2_0$ extends to an isometry $\wh\phi:\M^1\ra\M^2$.
\end{itemize} 
\end{theorem}

\begin{theorem}[Convergence of singular Ricci flows]
\label{thm_convergence_singular_ricci_flows}
Suppose $\{h_j\}$ is a sequence of smooth Riemannian metrics on a compact orientable $3$-manifold $M$, and $h_j\ra h_\infty$ smoothly as $j\ra\infty$.  For $j\in \n\cup \infty$, let $(\M^j,\t^j,\D_{\t^j},g^j)$ be a singular Ricci flow with time-$0$ slice $(M,h_j)$, and for every $T, C<\infty$, let 
$$
\M^j_{T,C}:=\{x\in \M^j\mid \t(x)\leq T, |{\Rm}|\leq C\}\,.
$$
Then there is a sequence $\{\M^\infty\supset U^j\stackrel{\Phi^j}{\lra}V^j\subset\M^j\}$ where:
\ben
\item $U^j$, $V^j$ are open, and $\Phi^j$ is a diffeomorphism.
\item For every $T$, $C$, we have $U^j\supset \M^\infty_{T,C}$,  $V^j\supset \M^j_{T,C}$. for large $j$. 
\item $\Phi^j$ is time-preserving, and the sequences $\{(\Phi^j)^*\D_{\t^j}\}$, $\{(\Phi^j)^*g^j\}$ converge smoothly on compact subsets $\M^\infty$ to $\D_{\t^\infty}$ and $g^\infty$, respectively. 
\een
\end{theorem}

If $M$ is a manifold, then a {\bf punctured copy of $M$} is a manifold diffeomorphic to $M\setminus S$, where $S\subset M$ is a finite (possibly empty) subset.  Note that if $M_1$, $M_2$ are compact $3$-manifolds, then punctured copies of $M_1$ and $M_2$ can be diffeomorphic only if $M_1$ is diffeomorphic to $M_2$.  This follows from the fact that if $D$, $D'$ are $3$-disks where $D'\subset \Int D$, then $\ol{D\setminus D'}$ is diffeomorphic to $S^2\times [0,1]$.   Hence the notion of ``filling in'' punctures is well-defined.  

The following result collects most of the topological and geometric properties of singular Ricci flows that will be needed in this paper.

\begin{theorem}[Structure of singular Ricci flows]
\label{thm_structure_singular_ricci_flow}
Let $(\M,\t,\D_\t,g)$ be an $(\eps_{\can},r)$-singular Ricci flow, where $\eps_{\can}$ is as in Theorem~\ref{thm_existence_uniqueness_singular_ricci_flow}.  Then:
\ben
\item For every $t\in [0,\infty)$, each component $C\subset\M_t$ is a punctured copy of some compact $3$-manifold.  
\item Let $\M_t^{\fill}$ be the (possibly empty) $3$-manifold obtained from $\M_t$ by filling in the punctures and throwing away the copies of $S^3$.  Then $\M_t^{\fill}$ is a compact $3$-manifold, i.e. all but finitely many components of $\M_t$ are punctured copies of $S^3$.  Furthermore, for every $t_1<t_2$ the  prime decomposition of $\M_{t_2}^{\fill}$ is part of the prime decomposition of $\M_{t_1}^{\fill}$.  Hence there are only finitely many times at which the  prime decomposition of $\M_t^{\fill}$ changes.
\item $\M_t^{\fill}$ is irreducible and aspherical for large $t$, depending only on the following bounds on the geometry of $\M_0$: 
 upper bounds on the curvature and volume, and a lower bound on the injectivity radius.
\item If the time-$0$ slice $\M_0$ is a spherical space form, then there is a time $\om\in[0,\infty)$ such that:
\begin{enumerate}[label=(\alph*)]
\item  For every $t<\om$, precisely one component $C_t$ of the time-$t$-slice $\M_t$ is a punctured copy of $\M_0$, and all other components are punctured copies of $S^3$.  
\item For every $t\geq\om$, the components of $\M_t$ are punctured $S^3$s.
\item If $\M_0$ is not diffeomorphic to $S^3$ or $RP^3$, then $C_t$ has no punctures for $t$ close to $\om(g)$ and the family of Riemannian manifolds
 $(C_t)_{t<\om}$ converges smoothly, modulo rescaling, to a manifold of constant sectional curvature $1$ as $t\ra \om$.  More precisely, if $t_0<\om(g)$ is close enough to $\om(g)$ that $C_{t_0}$ is compact and has positive sectional curvature, then:
	\ben
	\item The maximal product domain $\C$ with initial time slice $C_{t_0}$ is defined on $[t_0,\om(g))$, and is isometric to the Ricci flow spacetime of the (maximal) Ricci flow $(g_{t,t_0})_{t_0\in[t_0,\om(g))}$ on $C_{t_0}$ with initial condition $g_{t_0}$.  
	\item Modulo rescaling, the family of Riemannian metrics  $(g_{t,t_0})_{t_0\in[t_0,\om(g))}$ converges in the $C^\infty$-topology as $t\ra\om(g)$ to a metric $\ov g$ of constant sectional curvature $1$.
	\item For every $t\in[t_0,\om(g))$, the time slice $\C_t$ coincides with $C_t$. 
	\een 
\end{enumerate}
\item  If $\M_0$ is diffeomorphic to a closed hyperbolic $3$-manifold, then modulo rescaling the family of Riemannian manifolds $(\M_t)_{t\in [0,\infty)}$ converges smoothly to $\M_0$ equipped with a hyperbolic metric as $t\ra\infty$. 
More precisely,  if $t_0<\infty$, the time slice  $C_{t_0}$ is compact, and the maximal product domain $\C$ with initial time slice $C_{t_0}$ is defined on $[t_0,\infty)$, then:
\begin{enumerate}[label=(\roman*)]
	\item If $(g_{t,t_0})_{t_0\in[t_0,\infty)}$ is the maximal Ricci flow on $C_{t_0}$ with initial condition $g_{t_0}$ (which corresponds to $\C$ (cf. Lemma~\ref{lem_maximal_product_region}), then, modulo rescaling, the family of Riemannian metrics  $(g_{t,t_0})_{t_0\in[t_0,\infty)}$ converges in the $C^\infty$-topology as $t\ra\infty$ to a metric $\ov g$ of constant sectional curvature $-1$.
	\item For every $t\in[t_0,\infty)$, the time slice $\C_t$ coincides with $C_t$.
	\een
\een

\end{theorem}
\begin{proof}
The proof is a combination of known results.

(1)  is contained in \cite[Prop. 5.31]{kleiner_lott_singular_ricci_flows}.

We now prove (2).  Pick $0\leq t_1<t_2<\infty$.  Let $Y\subset \M_{t_2}$ be the union of finitely many connected components none of which is a punctured copy of $S^3$, and let $Y^{\fill}$ be the result of filling in the punctures of $Y$.  
By \cite[Theorem 1.13]{kleiner_lott_singular_ricci_flows}, there is a finite subset $S\subset Y$ such that flow of the time vector field $\D_\t$ is defined on $Y\setminus S$ over the time interval $[t_1-t_2,0]$, and hence it defines a smooth embedding $Y\setminus S\hookrightarrow \M_{t_1}$.  
Taking $t_1=0$, we see that the prime decomposition of $Y^{\fill}$ is part of the prime decomposition of $\M_0$, and hence the number of summands is bounded independently of the choice of $t_2$ and $Y$.  It follows that $\M_{t_2}^{\fill}$ (and similarly $\M_{t_1}^{\fill}$) are  compact $3$-manifolds, and without loss of generality we may assume that $Y$ is the union of all components of $\M_{t_2}$ that are not punctured copies of $S^3$.  The embedding $Y\hookrightarrow\M_{t_1}$ implies that the  prime decomposition of $\M_{t_2}^{\fill}$ is part of the prime decomposition of $\M_{t_1}^{\fill}$.  This proves (2).  

By \cite{colding_minicozzi_extinction,perelman_extinction}, there is a $\ul{t}<\infty$ such that for any Ricci flow with surgery (in the sense of Perelman \cite{perelman_surgery}) with sufficiently precise cutoff and starting from the Riemannian manifold $\M_0$, then for any $t\geq \ul{t}$ every component of the time-$t$ slice is irreducible and aspherical or a copy of $S^3$.  
Pick $t\geq \ul{t}$.  
Choose a connected component $C\subset \M_t$ that is not a punctured copy of $S^3$, and let $Z\subset C$ be a compact domain with spherical boundary components such that $\Int Z$ is diffeomorphic to $C$.  By the convergence theorem \cite[Thm. 1.2]{kleiner_lott_singular_ricci_flows} and \cite[Cor. 1.4]{bamler_kleiner_uniqueness_stability}, the domain $Z$ smoothly embeds in the time-$t$ slice of some Ricci flow with surgery starting from $\M_0$.  It follows that filling in the punctures of $C$, we get an irreducible and aspherical $3$-manifold.  This proves (3).

Suppose $\M_0$ is diffeomorphic to a spherical space form.  

If $\M_0$ is a copy of $S^3$, then by (2), for every $t\geq 0$ all components of $\M_t$ are punctured copies of $S^3$, and taking $\om =0$, assertions (a) and (b) follow.  

If $\M_0$ is not a copy of $S^3$, then by (2) and (3) there is an $\om\in (0,\infty)$ such that $\M_t^{\fill}$ is a copy of $\M_0$ for $t<\om$ and $\M_t^{\fill}=\emptyset$ for $t>\om$.  If $\M_\om^{\fill}$ is a copy of $\M_0$, then there is a compact domain with smooth boundary $Z\subset \M_\om$ such that $\Int Z$ is a punctured copy of $\M_0$.  Applying the flow of the time vector field $\D_\t$ for short time, we see that $Z$ embeds in $\M_t$ for some $t>\om$, contradicting $\M_t^{\fill}=\emptyset$.  Therefore $\M_\om^{\fill}=\emptyset$.  Hence assertions (4)(a) and (4)(b) hold. 

Now suppose $\M_0$ is not a copy of $S^3$ or $RP^3$.  Choose $t<\om$, and let $Z\subset \M_t$ be the component that is a punctured copy of $\M_0$.  Choose $\hat\eps>0$, and suppose $Z$ is not $\hat\eps$-close modulo rescaling to $\M_0$ equipped with a $K\equiv 1$ metric.  By  \cite[Prop. 5.31]{kleiner_lott_singular_ricci_flows}, there is a finite disjoint collection $\{N_i\}_{i=1}^k$ where:
\begin{enumerate}[label=(\roman*)]
\item Each $N_i$ is a domain with boundary in $Z$ which is diffeomorphic to one of the following:  $S^3$, $RP^3$, $S^2\times S^1$, $RP^3\# RP^3$, $S^2\times[0,\infty)$, $D^3$, $RP^3\setminus B^3$, or $S^2\times [0,1]$.  
\item $|{\Rm}|<Cr^{-2}(t)$ on the complement $Z\setminus \cup_i N_i$, where $C$ is a universal constant.  
\een
Since $\M_0$ is not a copy of $S^3$ or $RP^3$ and $Z$ is a punctured copy of $\M_0$, which is irreducible, each $N_i$ must be a copy of $S^2\times [0,\infty)$, $D^3$, or $S^2\times [0,1]$.  
The components of $Z\setminus \cup_i N_i$ embed in $\M_0$, and are therefore punctured copies of $S^3$ or $\M_0$.   
It follows that some component $W$ of $Z\setminus \cup_i N_i$ is a punctured copy of $\M_0$. 
Let $r_* > 0$ be a lower bound on $r(t)$ for $t \in [0, \omega)$.
By the canonical neighborhood assumption we can find a universal constant $C_*$ such that $|\partial_{\t} |{\Rm}| | < C_* r^{-2}_*$ whenever $|{\Rm}| \in ( r_*^{-2}, 2 r_*^{-2})$ (see \cite[Lemma 8.1]{bamler_kleiner_uniqueness_stability}).
Integrating this inequality, using the curvature bound in (ii) and $0$-completeness, it follows that there is a constant $\tau = \frac12 C_*^{-1} r_*^2 > 0$, which is independent of $t$, such that $W$ survives until time $t+\tau$ and such that $|{\Rm}| < 2 r_*^{-2}$ on $W (t')$ for all $t' \in [t, t+\tau]$ (compare \cite[Lemma 8.4]{bamler_kleiner_uniqueness_stability}).
So if $t$ is sufficiently close to $t$, then $W$ survives until time $\omega$, contradicting the fact that $\M_\om^{\fill}=\emptyset$. 
Hence modulo rescaling $Z$ is $\hat\eps$-close to $\M_0$ equipped with a metric of constant sectional curvature $1$. 
Thus the first part of assertion (4)(c) holds.

Now suppose that $C_{t_0}$ is compact and has positive sectional curvature for some $t_0<\om(g)$. Let $\C\subset\M_{[t_0,\infty)}$ be the maximal product domain with initial time slice $C_{t_0}$, and let $[t_0,t_+)$ be the time interval on which $\C$ is defined.  By Lemma~\ref{lem_maximal_product_region}, $\C$ is isometric to the spacetime for the  ordinary Ricci flow $(g_{t,t_0})_{t\in [t_0,t_+)}$ on $C_{t_0}$ with initial condition given by $g_{t_0}$.  Because $g_{t_0}$ has positive sectional curvature, it follows from \cite{hamilton_positive_ricci} that the metrics $(g_{t,t_0})_{t\in [t_0,t_+)}$ converge modulo rescaling to a $K\equiv 1$ metric $\ov g$ on $C_{t_0}$ as $t\ra t_+$.  If $t_+>\om(g)$, then $\M_t$ would contain a copy of $X$ for some $t>\om(g)$, contradicting  assertion (4)(b).  If $t_+<\om(g)$, then by the compactness of $C_{t_+}$, we may flow $C_{t_+}$ backward and forward under $\D_\t$ for a short time.  By (4)(a) this yields $C_t$ for $t$ close to $t_+$, and in particular $\C_t$ for $t<t_+$ close to $t_+$.  Hence $\C$ may be extended past $t_+$, which is a contradiction.  Thus $t_+=\om(g)$.  Hence (4)(c)(i) and (4)(c)(ii) hold.  Assertion (4)(c)(iii) follows from (4)(a).

Now suppose $\M_0$ is diffeomorphic to a hyperbolic $3$-manifold. 

Let $(\M^j,\t^j,\D_{\t^j},g^j)\ra\M^\infty$ be a convergent sequence of Ricci flow spacetimes as in \cite[Thm. 1.2]{kleiner_lott_singular_ricci_flows}, i.e. the $\M^j$s are associated with a sequence of Ricci flows with surgery with initial conditions isometric to $\M_0$, where the Perelman surgery parameter $\de_j$ tends to zero as $j\ra \infty$.  By  \cite[Cor. 1.4]{bamler_kleiner_uniqueness_stability} we may take $\M^\infty=\M$.  We recall that Perelman's work implies that for fixed $j$, a statement similar to (5) holds: as $t\ra\infty$ i.e. the Riemannian manifolds $(\M^j_t,t^{-2}g^j_t)$ converge in the smooth topology to $\M_0$ with a hyperbolic metric \cite{perelman_surgery,bamler_certain_topologies}.  Although Perelman's argument by itself does not provide the uniform control needed to pass the statement to the limit to obtain (5) directly, his argument can nonetheless be implemented by using some results from \cite{kleiner_lott_singular_ricci_flows}, as we now explain.

\begin{claim*}
For every $j\in \n\cup\{\infty\}$, and every $t\in [0,\infty)$:
\begin{itemize}
\item The scalar curvature attains a nonpositive minimum $R^j_{\min}(t)$ on $\M^j_t$.
\item the function $t\mapsto \wh R^j(t):=R^j_{\min}(t)(V^j)^\frac23(t)$ is nondecreasing, where $V^j(t)$ is the volume of $\M^j_t$.  
\end{itemize}
Moreover, $\wh R^j\ra \wh R^\infty$ uniformly on compact sets as $j\ra\infty$.
\end{claim*}
\begin{proof}
For $j\in \n$, the fact that $R^j_{\min}(t)$ is well-defined and negative, and $t\mapsto \wh R^j(t)$ is nondecreasing, was shown by Perelman.  From the properness asserted in \cite[Thm. 1.2(a)]{kleiner_lott_singular_ricci_flows} it follows that $t\mapsto R^\infty_{\min}(t)$ is a well-defined continuous function, and the main assertion of \cite[Thm. 1.2]{kleiner_lott_singular_ricci_flows} then implies that $R^j_{\min}\ra R^\infty_{\min}$ uniformly on compact sets.  

Since $V^\infty$ is continuous and $V^j\ra V^\infty$ uniformly on compact sets by \cite[Thm.~4.1, Cor. 7.11]{kleiner_lott_singular_ricci_flows}, it follows that $\wh R^j\ra \wh R^\infty$ uniformly on compact sets.  Because $\wh R^\infty$ is a uniform limit of nondecreasing functions, it is also nondecreasing.
\end{proof}

\bigskip

Using the fact that $\wh R^j\ra \wh R^\infty$ uniformly on compact sets, the arguments from \cite[Secs. 89-91]{kleiner_lott_perelman_notes} may now be implemented uniformly for all the Ricci flows with surgery, with the slight modification that the ``slowly varying almost hyperbolic structures'' in  \cite[Prop. 90.1]{kleiner_lott_perelman_notes}  are only defined on an interval $[T_0,T_j]$ for some sequence $T_j\ra \infty$.  After passing to a subsequence, we may use \cite[Thm. 1.2]{kleiner_lott_singular_ricci_flows} to obtain a version of \cite[Prop. 90.1]{kleiner_lott_perelman_notes} for $\M=\M^\infty$.   The proof of incompressibility of the cuspidal tori as in \cite[Sec. 91]{kleiner_lott_perelman_notes} carries over to the singular Ricci flow $\M$, since compressing disks avoid the thin parts of $\M$.  Alternatively, one may deduce incompressibility in time slices of $\M^\infty$ from incompressibility in $\M^j$ for large $j$.  The complement of the (truncated) hyperbolic regions consists of graph manifolds \cite[Sec. 91]{kleiner_lott_perelman_notes}.  By (2) we conclude that there is precisely one hyperbolic component, and it coincides with $\M^\infty_t$ for large $t$.  
Thus we have proven the first part of assertion (5).

The proof of (5)(i)--(ii) is similar to the proof of (4)(c)(i)--(iii), except that instead of appealing to \cite{hamilton_positive_ricci}, we use the convergence of normalized Ricci flow shown in \cite{ye_convergence,bamler_stability_hyperbolic_cusps}.
\end{proof}



\section{The canonical limiting constant curvature metric}
\label{sec_canonical_limiting_metric}
In this section, we let $X$ be a $3$-dimensional spherical space form other than $S^3$ or $RP^3$.   Recall that $3$-dimensional spherical space forms are always orientable.

By Theorem~\ref{thm_structure_singular_ricci_flow},  a singular Ricci flow $\M$ starting from an arbitrary metric $g\in \met(X)$, contains a family of time slice components whose metrics converge, modulo rescaling, to a round metric.  
In this section, we will use a result from \cite{kleiner_lott_singular_ricci_flows} to flow these metrics back to the time $0$ slice $\M_0\simeq X$, with the caveat that the flow is not  defined everywhere -- it is only defined on the complement of a finite subset.  Using the uniqueness and continuity theorems of \cite{bamler_kleiner_uniqueness_stability}, we show that this process yields a canonical partially defined metric of constant of sectional curvature $1$ on $X$, which depends continuously on $g$ in the sense of Definition~\ref{def_topology_partially_defined_metrics}.  See Figure~\ref{fig_spacetime} for a depiction of this.

Pick $g\in \met(X)$ and let $\M$ be a singular Ricci flow with $\M_0=(X,g)$.  Let $\om(g)=\om$ be as in Theorem~\ref{thm_structure_singular_ricci_flow}, and for every $t<\om(g)$ let $C_t$ be the unique component of $\M_t$ that is a punctured copy of $X$.  

For every $t_1,t_2\in[0,\om(g))$, let $C_{t_1,t_2}\subset C_{t_1}$ be the set of points in $C_{t_1}$ that survive until time $t_2$, i.e. the points for which the time $(t_2-t_1)$-flow of the time vector field $\D_\t$ is defined.  Then $C_{t_1,t_2}$ is an open subset of $C_{t_1}$, and the time $(t_2-t_1)$-flow of $\D_\t$ defines a smooth map $\Phi_{t_1,t_2}:C_{t_1,t_2}\ra \M_{t_2}$, which is a diffeomorphism onto its image.  We define a metric $g_{t_1,t_2}$ on $\Phi_{t_1,t_2}(C_{t_1,t_2})$ by $g_{t_1,t_2}:=(\Phi_{t_1,t_2})_*g_{t_1}$, where $g_{t_1}$ is the spacetime metric on $C_{t_1,t_2}\subset \M_{t_1}$.  We let $W_g(t):=\Phi_{t,0}(C_{t,0})$, so  $g_{t,0}$ is a metric on $W_g(t)$.

\begin{lemma}[Limiting $K\equiv 1$ metric]
\label{lem_canonical_limiting_metric}
Choose $t_0<\om(g)$ such that $C_{t_0}$ is compact and has positive sectional curvature; such a time $t_0$ exists by Theorem~\ref{thm_structure_singular_ricci_flow}(4)(a).  Then:  
\ben
\item $W_g(t)=W_g(t_0)=:W_g$ for all $t\in[t_0,\om(g))$.
\item Modulo rescaling, $g_{t,0}$ converges in the smooth topology to a $K\equiv 1$ metric $\check g$ on $W_g$ as $t\ra \om(g)$.
\item $(W_g,\check g)$ is isometric to $(X\setminus S, g_X)$ for some finite (possibly empty) subset $S\subset X$, where the cardinality of $S$ is bounded above by a constant depending only on bounds on the curvature, injectivity radius, and volume of $g$.
\item If $g$ has constant sectional curvature, then $W_g=X$ and $\check g=\lambda g$ for some $\lambda\in(0,\infty)$.
\een
\end{lemma}
We caution the reader that although $(W_g,\check g)$ is isometric to $(X,g_X)$ with finitely many points removed, the complement $X\setminus W_g$ may have nonempty interior and could, in principle, be quite irregular.

\begin{proof}[Proof of Lemma~\ref{lem_canonical_limiting_metric}]
Let $\C\subset\M_{[t_0,\infty)}$ be the maximal product domain with  $\C_{t_0}=C_{t_0}$.  By Theorem~\ref{thm_structure_singular_ricci_flow}(4)(c), the domain $\C$ is defined on the time interval $[t_0,\om(g))$.
Hence for all $t\in[t_0,\om(g))$ we have $C_{t,t_0}=C_{t}$,  $C_{t,0}=\Phi^{-1}_{t,t_0}(C_{t_0,0})$, $\Phi_{t,0}=\Phi_{t_0,0}\circ\Phi_{t,t_0}$ and $W_g(t)=\Phi_{t,0}(C_{t,0})=\Phi_{t_0,0}(C_{t_0,0})=W_g(t_0)$.   Thus (1) holds.

Note that on $\Phi_{t,0}(C_{t,0})$
$$
g_{t,0}=(\Phi_{t,0})_*g_t=(\Phi_{t_0,0})_*(\Phi_{t,t_0})_*(g_t)=(\Phi_{t_0,0})_*g_{t,t_0}\,.
$$
By Theorem~\ref{thm_structure_singular_ricci_flow}(4), modulo rescaling, $(g_{t,t_0})_{t\in[t_0,\om(g))}$  converges in $\met(C_{t_0,0})$ to the $K\equiv 1$ metric $\ov g$ on $C_{t_0,0}$ as $t\ra\om(g)$.  Consequently, modulo rescaling,  $(g_{t,0})_{t\in[t_0,\om(g))}$ converges in $\met(W_g)$ to the $K\equiv 1$ metric $(\Phi_{t_0,0})_*\ov g$.
 Thus assertion (2) holds. 

By \cite[Theorem 1.13]{kleiner_lott_singular_ricci_flows}, all but finitely many points in $C_{t_0}$ survive until $t=0$, i.e. $C_{t_0}\setminus C_{t_0,0}$ is finite; moreover the cardinality is bounded above depending only on the volume of $\M_0$ and $\om(g)$.  
There is a unique $K\equiv 1$ metric on $X$ up to isometry \cite{milnor_whitehead_torsion}, so $(C_{t_0},\ov g)$ is isometric to $(X,g_X)$.  Since $(W_g,\check g)$ is isometric to $(C_{t_0,0},\ov g)$, assertion (3) holds.

Now suppose $g$ has constant sectional curvature.  Then $\M$ is the product Ricci flow spacetime corresponding to a shrinking round space form. Hence $W_g=X$, and $g_{t,0}$ agrees with $g$ modulo rescaling, for all $t\in [0,\om(g))$, so $\check g$ agrees with $g$ modulo rescaling, and assertion (4) holds.
\end{proof}

\bigskip

By Theorem~\ref{thm_structure_singular_ricci_flow}, for every  $g\in\met(X)$ there exists a singular Ricci flow $\M$ with $\M_0$ isometric to $(X,g)$ which is unique up to isometry; hence the pair $(W_g,\check g)$ is also independent of the choice of $\M$, i.e. it is a well-defined invariant of $g\in \met(X)$.  

\bigskip

Next, we show that the pair $(W_g,\check g)$ varies continuously with $g$, in the sense of Definition~\ref{def_topology_partially_defined_metrics}.  After unwinding definitions, this is a straightforward consequence of the convergence theorem of \cite{bamler_kleiner_uniqueness_stability}.

\begin{lemma}
\label{lem_continuous_partmet}
The assignment $g\mapsto (W_g,\check g)$ defines a continuous map $\met(X)\ra \partmet(X)$.  
\end{lemma}
\begin{proof}
Suppose $g^j\ra g^\infty$ in $\met(X)$.  For all $j\in \n\cup \{\infty\}$, let $\M^j$ be a singular Ricci flow with $\M^j_0=(X,g^j)$, let $\om(g^j)$,  $C^j_t$, be as in Theorem~\ref{thm_structure_singular_ricci_flow}, and $(W_{g^j},\check g^j)$  be as in Corollary~\ref{lem_canonical_limiting_metric}.  Choose a compact subset $Z\subset W_{g^\infty}$.  To prove the lemma, we will show that $Z\subset W_{g^j}$ for large $j$, and that $\check g^j \ra \check g^\infty$  in the $C^\infty$-topology, on an open subset containing $Z$.

Let $\phi^j:\M^\infty\supset U^j\ra V^j\subset \M^j$ be the diffeomorphism onto its image from Theorem~\ref{thm_convergence_singular_ricci_flows}.

Let $t_0$ be the time from Lemma~\ref{lem_canonical_limiting_metric} for $\M^\infty$.  By Theorem~\ref{thm_convergence_singular_ricci_flows},  the map $\phi^j$ is defined on $C^\infty_{t_0}$ for large $j$, and  $(\phi^j_{t_0})^*g^j\ra g^\infty$ on $C^\infty_{t_0}$.  
Therefore, without loss of generality, we may assume that $\phi^j$ is defined on $C^\infty_{t_0}$ for all $j$, and $C^j_{t_0}$ has positive sectional curvature.  For $j\in \n\cup\{\infty\}$, if $\C^j\subset\M^j$ is the maximal product domain with time-$t_0$ slice $C^j_{t_0}$, then $\C^j$ is defined on $[t_0,\om(g^j))$ and $\C^j_t=C^j_t$ for all $t\in[t_0,\om(g^j))$,  by  Theorem~\ref{thm_structure_singular_ricci_flow}(4).  

The flow $\Phi^j_{t,t_0}$ of the time vector field $\D_{\t^j}$ is defined on $C^j_t$ for $t\in [t_0,\om(g^j))$, $j\in \n\cup\{\infty\}$, since $\C^j$ is a product domain  in $[t_0,\om(g^j))$. Thus, for $j\in \n\cup\{\infty\}$, $g^j_{t,t_0}:=(\Phi^j_{t,t_0})_*g^j$ is a well-defined smooth metric on $C^j_{t_0}$ for all $t\in[t_0,\om(g^j))$, and modulo rescaling, $g^j_{t,t_0}\ra\ov g^j\in \met_{K\equiv 1}(C^j_{t_0})$ as $t\ra\om(g^j)$.  By \cite{knopf_et_al}, since $(\phi^j)^*g^j\ra g^\infty$ on $C^\infty_{t_0}$ in the $C^\infty$-topology, it follows that $(\phi^j)^*\ov g^j\ra \ov g^\infty$ in the $C^\infty$-topology.

By assumption $Z\subset W_{g^\infty}$, so it survives until time $t_0$ in $\M^\infty$, and $\Phi^\infty_{0,t_0}(Z)\subset C^\infty_{t_0}$.  Since $Z$ is compact, there is a product domain 
$\N^\infty\subset\M^\infty$ defined on $[0,t_0]$, such that $\N^\infty$ has compact closure in $\M^\infty$,  and the time zero slice $\N^\infty_0$ is an open subset of $W_{g^\infty}$ containing $Z$.

For large $j$ the map $\phi^j$ is defined on $\N^\infty$, and we let $\wh\N^j$ be the pullback of $\M^j$ under $\phi^j\restr\N^\infty$, i.e. $\wh\N^j:=(\N^\infty,\wh\t^j=\t^\infty,\D_{\wh t^j},\wh g^j)$ where $\D_{\wh t^j}:=(\phi^j\restr \N^\infty)^*\D_{\t^j}$, $\wh g^j:=(\phi^j\restr\N^\infty)^*g^j$. 
(Note that $\wh\N^j$ is not quite a Ricci flow spacetime because it has boundary points in the time-$t_0$ slice.)
By Theorem~\ref{thm_convergence_singular_ricci_flows}, we have $\D_{\wh t^j}\ra \D_{\t^\infty}$, $\wh g^j\ra g^\infty$ in the $C^\infty$-topology on $\N^\infty$ as $j\ra\infty$.  

Now choose an open set $Z'\subset\N^\infty_0$ with compact closure in $\N^\infty_0$, such that $Z\subset Z'$.  Since $\D_{\wh \t^j}$ converges to the product vector field $\D_{\t^\infty}$ on $\N^\infty$, it follows that for large $j$, the flow $\wh\Phi^j_{0,t_0}\restr Z'$ of $\D_{\wh \t^j}$ is defined and takes values in $C^\infty_{t_0}$, and $\wh\Phi^j_{0,t_0}\restr Z'\ra \Phi^\infty_{0,t_0}\restr Z'$ in the $C^\infty$-topology as $j\ra\infty$.  Because $(\phi^j_{t_0})^*\ov g^j\ra \ov g^\infty$ as $j\ra\infty$, we get 
\begin{equation}
\label{eqn_phi_pullback_g_bar}
\left(\wh\Phi^j_{0,t_0}\right)^*(\phi^j_{t_0})^*\ov g^j\ra (\Phi^\infty_{0,t_0})^*\ov g^\infty=\check g^\infty \qquad\text{on}\qquad Z'
\end{equation} 
in the $C^\infty$ topology, as $j\ra\infty$.   

The map $\phi^j:\wh\N^j \ra \phi^j(\N^\infty)\subset \M^j$ preserves time functions, time vector fields, and metrics.   Therefore, the assertions in the previous paragraph imply that for large $j$, $\Phi^j_{0,t_0}$ is defined on $Z'$  and takes values in $\phi^j_{t_0}(C^\infty_{t_0})=C^j_{t_0}$ and  
$\check g^j\restr Z'=(\Phi^j_{0,t_0}\restr Z')^*\ov g^j=(\phi^j_0)^*(\Phi^j_{0,t_0}\restr Z')^*\ov g^j=(\wh\Phi^j_{0,t_0}\restr Z')^*(\phi^j_{t_0})^*\ov g^j$.  By (\ref{eqn_phi_pullback_g_bar}) we conclude that $\check g^j \restr Z'\ra \check g^\infty\restr Z'$ in the $C^\infty$- topology as $j\ra \infty$.
\end{proof}

\begin{remark}
\label{rem_alternate_continuous_dependence}
In order to prove the continuous dependence of the limiting round metric on the initial metric of positive sectional curvature, we invoked the continuity theorem from \cite{knopf_et_al}.  However, in our applications in the next section, it would work equally well if instead of using the limiting $K\equiv 1$ metric produced by Ricci flow, we used some other geometric construction to replace a metric $h$ that is $\eps$-close to a round Riemannian metric with a round metric $\ov h$, as long as $\ov h$ depends continuously on $h$ in the smooth topology and is equivariant with respect to diffeomorphisms.  For instance, one could use the smooth dependence of certain eigenspaces of the Laplacian.
\end{remark}

\section{Extending constant curvature metrics}
\label{sec_extending_constant_curvature_metrics}
The goal of this section is the following proposition, which asserts that under certain conditions a finite dimensional continuous family of partially defined $K\equiv 1$ metrics can be extended to a continuous family of globally defined $K\equiv 1$ metrics. 

In this section $X$ will denote  a spherical space form other than $S^3$.  Pick $g_X\in\met_{K\equiv 1}(X)$. We recall  that $g_X$ is unique up to isometry \cite{milnor_whitehead_torsion,de_rham_torsion,franz_torsion,reidemeister_torsion}.

In the following we will use the term {\bf polyhedron} to refer to (the geometric realization of) a simplicial complex.  

\begin{proposition}[Extending $K\equiv 1$ metrics]
\label{prop_extending_metrics_general}
Let $P_0$ be a finite polyhedron, and $Q_0\subset P_0$ a subpolyhedron.  Suppose $P_0\ni p\mapsto (W_p,g(p))$ is an assignment with the following properties:
\begin{enumerate}[label=(\roman*)]
\item $p\mapsto (W_p,g(p))$ defines a continuous map $P_0\ra \partmet(X)$. 
\item There is an $n<\infty$ such that for every $p\in P_0$ the Riemannian manifold $(W_p,g(p))$ is isometric to $(X\setminus S_p,g_X)$ for a finite subset $S_p\subset X$ with $|S_p|\leq n$.   
\item $W_q=X$ for all $q\in Q_0$.
\een 
Then there is a continuous map $\wh g:P_0\ra \met_{K\equiv 1}(X)$ such that for every $q\in Q_0$ we have $\wh g(q)=g(q)$.
\end{proposition}
We remark that we will only apply Proposition~\ref{prop_extending_metrics_general} in the case when $P_0$ is an $(m+1)$-disk and $Q_0=\D P_0$ is its boundary $m$-sphere, for some $m\geq 0$.

Before proceeding with the proof, we first provide some motivation for the proof.   

Take $p\in P_0$, and consider the open subset $W_p\subset X$.  Since $(W_p,g(p))$ is isometric to $(X\setminus S_p,g_X)$, the ends of $W_p$ are diffeomorphic to $S^2\times [0,\infty)$, and hence there is a compact domain with smooth boundary $Z_p\subset W_p$ with $2$-sphere boundary components, such that $\Int Z_p$ is diffeomorphic to $X\setminus S_p$.  By a simple topological argument, the closure of $X\setminus Z_p$ is a union of a disjoint collection $\C_p$ of $3$-disks.  It is not hard to see that for every $3$-disk $Y\in \C_p$, the restriction of $g(p)$ to a small neighborhood of $\D Y$ in $Y$ extends to a metric on $Y$ with $K\equiv 1$.  Moreover, by using Hatcher's theorem, one can see that the extension is unique, up to contractible ambiguity. Combining the extensions for each $Y\in \C_p$, we obtain an extension of $g(p)\restr Z_p$ to $X$, which is also unique up to contractible ambiguity.  

To adapt the preceding observations into a proof of the proposition, we first choose a fine subdivision $P$ of $P_0$, and for every face $\si$ of $P$, we choose a domain with smooth boundary $Z_\si\subset X$ such that (among other things) $ Z_\si\subset W_p$ for every $p\in \si$, and $Z_\si=X$ if $p\in Q_0$.  We then prove the proposition by extending $g$ inductively over the skeleta of $P$.  In the induction step, we assume that $\wh g$ has been defined on the $m$-skeleton of $P$, and then we extend it to an $(m+1)$-face $\si\subset P$, such that it agrees with $g$ on $Z_\si$.  Since the closure of $X\setminus Z_\si$ is a disjoint collection of $3$-disks, our problem reduces  to solving an extension problem for $3$-disks, which may be deduced from Hatcher's theorem.

For the remainder of the section we fix an assignment $P_0\ni p\mapsto (W_p,g(p))$ as in the statement of Proposition~\ref{prop_extending_metrics_general}.

The first step in the proof of Proposition~\ref{prop_extending_metrics_general} is to define the subdivision $P$ and the collection $\{Z_\si\}_{\si\in\face P}$ described above.

\begin{lemma}
\label{lem_combinatorial_assignments}
There is a subdivision $P$ of $P_0$, and to every face $\si$ of $P$ we can assign  a triple $(Z_\si,\C_\si,U_\si)$ with the following properties:
\begin{enumerate}[label=(\alph*)]
\item For every face $\si$ of $P$, $\C_\si$ is a finite disjoint collection of $3$-disks in $X$, and  $Z_\si=\ol{X\setminus \cup_{Y\in\C_\si}Y}$.
\item For every strict inclusion of faces $\si_1\subsetneq \si_2$ we have $Z_{\si_2}\subset \Int(Z_{\si_1})$. 
\item For every face $\si$ of $P$, and every $p\in \si$, we have:
	\begin{enumerate}[label=(\roman*)]
	\item $U_\si$ is an open subset of $X$ with $Z_\si\subset U_\si\subset W_p$. 
	\item $(U_\si\cap Y,g(p)) $ isometrically embeds in $(S^3,g_{S^3})$ for every $Y\in \C_\si$.  
	\item If $\si\cap Q_0\neq\emptyset$, then $Z_\si=X$.
	\een
\een
\end{lemma}

To prove Proposition~\ref{prop_extending_metrics_general} using the combinatorial structure from Lemma~\ref{lem_combinatorial_assignments}, the main ingredient is the following local extension lemma, which is based on  (a version of) the Smale conjecture, as proved by Hatcher \cite{hatcher_smale_conjecture}.  

\begin{lemma}[Extending $K\equiv 1$ metrics over a ball]
\label{lem_extending_metrics_ball}
In the following, we let $S^2$ and $D^3$ denote the unit sphere and unit disk in $\R^3$, respectively, and we let $N_r(S^2)$ denote the metric $r$-neighborhood of $S^2\subset \R^3$.

Suppose $m\geq 0$, $\rho>0$ and: 
\begin{enumerate}[label=(\roman*)]
\item $h_{m+1}:D^{m+1}\ra \met_{K\equiv 1}(N_{\rho}(S^2)\cap D^3)$ is a continuous map such that for all $p\in D^{m+1}$, the Riemannian manifold $(N_{\rho}(S^2)\cap D^3,h_{m+1}(p))$ isometrically embeds in $(S^3,g_{S^3})$.  Here $\met_{K\equiv 1}(N_{\rho}(S^2)\cap D^3)$ is equipped with the $C^\infty_{\loc}$-topology.
\item $\wh h_m:S^m\ra\met_{K\equiv 1}(D^3)$ is a continuous map such that for every $p\in S^m$ we have $\wh h_m(p)=h_{m+1}(p)$ on $N_{\rho}(S^2)\cap D^3$, and $(D^3,h_m(p))$ isometrically embeds in $(S^3,g_{S^3})$.
\een

Then, after shrinking $\rho$ if necessary, there is a continuous map $\wh h_{m+1}:D^{m+1}\ra \met_{K\equiv 1}(D^3)$ such that:
\begin{enumerate}[label=(\alph*)]
\item $\wh h_{m+1}(p)=h_{m+1}(p)$ on $N_{\rho}(S^2)\cap D^3$ for all $p\in D^{m+1}$.
\item $\wh h_{m+1}(p)=h_m(p)$ for all $p\in S^m$.
\een
\end{lemma}

\bigskip
\begin{proof}[Proof of Proposition~\ref{prop_extending_metrics_general} assuming Lemmas~\ref{lem_combinatorial_assignments} and \ref{lem_extending_metrics_ball}]
Let $P$ and the assignment $\face P\ni \si\mapsto(Z_\si,\C_\si,U_\si)$ be as in Lemma~\ref{lem_combinatorial_assignments}. Let $Q$ be the corresponding subdivision of $Q_0$.

Pick $m+1\geq 0$ with $m+1\leq \dim (P\setminus Q)$.  Assume inductively that if $m\geq 0$, then we have defined a continuous map $\wh g_m:P^{(m)}\cup Q\ra \met_{K\equiv 1}(X)$ such that for every face $\tau\subset P^{(m)}\cup Q$, and every $p\in \tau$, the metric $\wh g_m(p)$ agrees with $g(p)$ on $Z_\tau\subset W_p$.

Pick an $(m+1)$-face $\si\subset P^{(m+1)}$ with $\si\not\subset Q$, and choose $Y\in \C_\si$.  We wish to  apply Lemma~\ref{lem_extending_metrics_ball}.  To that end we choose a diffeomorphism $\al: D^3\ra Y$, and a homeomorphism $\be:D^{m+1}\ra \si$.  If $\rho>0$ is sufficiently small, then $\al(N_\rho(S^2)\cap D^3)\subset \cap_{\tau\subsetneq\si} Z_\tau$ by assertion (b) of Lemma~\ref{lem_combinatorial_assignments}.  Since $\cap_{\tau\subsetneq \si}Z_\tau\subset U_\si\subset\cap_{p\in \D\si}W_p$ by assertion (c) of Lemma~\ref{lem_combinatorial_assignments}, we may define continuous maps
$$
h_{m+1}:D^{m+1}\lra \met_{K\equiv 1}(N_\rho(S^2)\cap D^3)\,,\qquad \wh h_m:S^m\lra \met_{K\equiv 1}(D^3)
$$
by $h_{m+1}(p):=(\al\restr N_\rho(S^2)\cap D^3)^*g(\beta(p))$, $\wh h_m(p):=\al^*\wh g_m(\beta(p))$.    Note that for every $p\in S^m$, the Riemannian manifold $(D^3,\wh h_m(p))$ isometrically embeds in $(X,g_X)$ by construction, and since $D^3$ is simply connected, this embedding may be lifted to an isometric embedding $(D^3,\wh h_m(p))\ra (S^3,g_{S^3})$.  Similarly, for every $p\in D^{m+1}$, the manifold $(N_\rho(S^2)\cap D^3,h_{m+1}(p))$ isometrically embeds in $(S^3,g_{S^3})$ by assertion (c)(ii) of Lemma~\ref{lem_combinatorial_assignments}.

Applying Lemma~\ref{lem_extending_metrics_ball}, after shrinking $\rho$, we obtain a continuous map $\wh h_{m+1}:D^{m+1}\ra \met_{K\equiv 1}(D^3)$ such that $\wh h_{m+1}(p)=\wh h_m(p)$ for all $p\in S^m$, $\wh h_{m+1}(p)=h_{m+1}(p)$ on $N_\rho(S^2)\cap D^3$ for all $p\in D^{m+1}$.  Now let $g_{\si,Y}:\si\ra \met_{K\equiv 1}(Y)$ be given by $g_{\si,Y}(\be(p))=\al_*\wh h_{m+1}(p)$.  

We may extend $\wh g_m$ to a continuous map $\wh g_{m+1}:P^{(m+1)}\cup Q\ra\met_{K\equiv 1}(X)$ by letting $\wh g_{m+1}\restr\si$ agree with $g_{\si,Y}$ on $Y$ for each $Y \in \C_\sigma$.  Note that by construction, for every face $\si$ of $P^{(m+1)}$ and every $p\in \si$, $\wh g_{m+1}(p)$ agrees with $g(p)$ on $Z_\si$, and hence on $X$ if $p\in Q$, by assertion (c)(iii) of Lemma~\ref{lem_combinatorial_assignments}.

By induction we obtain the desired map $\wh g:P\ra\met_{K\equiv 1}(X)$. 
\end{proof}

\bigskip\bigskip
We now prove Lemmas~\ref{lem_combinatorial_assignments} and \ref{lem_extending_metrics_ball}.  Before proving Lemma~\ref{lem_combinatorial_assignments}, we need a preparatory result.

\begin{lemma}
\label{lem_preparatory_stuff}

\mbox{}
\ben
\item Let  $\C$ be a finite disjoint collection of $3$-disks in $X$, and $Z\subset X$ be a domain with smooth boundary diffeomorphic to $\ol{X\setminus\cup_{Y\in \C}Y}$.   Then there is a (unique) finite disjoint collection of $3$-disks $\C'$ such that $Z=\ol{X\setminus\cup_{Y\in \C'}Y}$.
\item Let $\C$ be a collection of at most $j$ closed balls of radius at most $ r$ in a metric space $Z$.  Then there is a disjoint collection $\C'$ of at most $j$ closed balls of radius $<4^jr$ such that $\{\Int Y\}_{Y\in \C'}$ covers $\cup_{Y\in \C}Y$.
\een
\end{lemma}
\begin{proof}
(1).  Recall that $X$ is irreducible and not diffeomorphic to $S^3$,  so every embedded $2$-sphere in $X$ bounds two domains with boundary, precisely one of which is a $3$-disk.  

Let $\C'$ be the collection of closures of the components of $X\setminus Z$.  Suppose  that $Y_0\in\C'$ is not diffeomorphic to a $3$-disk.  Then $X\setminus\Int Y_0$ is a $3$-disk, and hence $Z\subset X\setminus\Int Y_0$ embeds in $S^3$.  But then the embedding $Z\ra S^3$ may be extended to a diffeomorphism $X\ra S^3$ by extending over the $3$-ball components of $X\setminus Z$; this is a contradiction. Therefore $\C'$ is a collection of $3$-disks, and assertion (1) follows.

(2).  This follows by induction on $j$, by replacing a pair of balls $\ol{B(x_1,r_1)}$, $\ol{B(x_2,r_2)}$ with $r_1\leq r_2$ such that $\ol{B(x_1,r_1)}\cap \ol{B(x_2,r_2)}\neq \emptyset$ with $\ol{B(x_2,4r_2)}$.
\end{proof}

\bigskip
\begin{proof}[Proof of Lemma~\ref{lem_combinatorial_assignments}]
Let $r_X>0$ be injectivity radius of $(X,g_X)$, and let $r\in(0,r_X)$ be a constant to be determined later.

For $p\in D^{m+1}$ let $d_p$ be the Riemannian distance function for $(W_p,g(p))$, and let $(\ov W_p,\ov d_p)$ be the completion of the metric space $(W_p,d_p)$.  By assumption (ii) of Proposition~\ref{prop_extending_metrics_general}, we know that $(\ov W_p,\ov d_p)$ is isometric to $(X,d_X)$.  

For each $p\in P_0$, we let $Z^0_p:=\ov W_p\setminus \cup_{x\in \ov W_p\setminus W_p}B_{\ov d_p}(x,r)$, so $Z^0_p\subset W_p\subset X$.  By the definition of the topology on $\partmet(X)$ and condition (i) of Proposition~\ref{prop_extending_metrics_general}, for all $p\in P_0$ there is an open subset $V_p\subset P_0$ such that for all $p'\in V_p$, we have $Z^0_p\subset W_{p'}$, and $2^{-1}g(p)< g(p')< 2g(p)$ on $Z^0_p$.  Let $P$ be a subdivision of $P_0$ such that the closed star cover $\{\ol{\St}(v,P)\}_{v\in P^{(0)}}$ refines the cover $\{V_p\}_{p\in P_0}$.  For every vertex $v\in P^{(0)}$ choose $p_v\in P$ such that $V_{p_v}\supset \ol{\St}(v,P)$, and let $Z^1_v:=Z^0_{p_v}$.  After refining $P$ if necessary, we may assume $Z^1_v=X$  for every vertex $v$ with $\ol{\St}(v,P)\cap Q\neq\emptyset$.  

For every face $\si$ of $P$, choose $b_\si\in \si$, and let $Z^1_\si:=\cup_{v\in \si^{(0)}}Z^1_v$.  Note that if $p\in\si$, then $Z^1_\si\subset W_p$, and $2^{-1}g(b_p)< g(p)< 2 g(b_p)$ on $Z^1_\si$.  

\begin{claim*} 
(Let $n$ be the constant from Proposition~\ref{prop_extending_metrics_general}(ii)). There exist universal constants $\{c_k\}_{k\geq 0}$, $\{c_k'(n)\}_{k\geq 0}$, $\{\ov r(k,n)\}_{k\geq 0}$  such that if $r\leq \ov r(\dim P,n)$, then there are  collections $\{\C_\si\}_{\si\in\face P}$ such that:
\begin{enumerate}[label=(\alph*)]
	\item $\C_\si$ is a disjoint collection of $3$-disks in $X$ for every $\si\in\face P$.
	\item   $\{\Int Y\}_{Y\in\C_\si}$ covers $X\setminus Z^1_\si$ for every $ \si\in\face P$.  
	\item For every $k$-face $\si$ of $P$, $|\C_\si|\leq c_kn$ and for every $ Y\in \C_\si$, the boundary $2$-sphere $\D Y$ has intrinsic diameter $\leq c_k'(n) r$ w.r.t. the Riemannian metric $g(b_\si)$.  Note that $g(b_\si)$ is defined on $\D Y$ since (b) implies $\D Y\subset  Z^1_\si\subset W_{b_\si}$.
	\item For all $ \si_1\subsetneq\si_2$ the collection $\{\Int Y\}_{Y\in \C_{\si_2}}$ covers $\cup_{Y\in \C_{\si_1}}Y$.  
	\item If $\si\in\face P$ and $W_{p_v}=X$ for every $v\in \si^{(0)}$, then $\C_\si=\emptyset$.
	\een
\end{claim*}
\begin{proof}
By induction on $m$ we will prove the existence the constants $\{c_k\}_{k\leq m}$, $\{c_k'(n)\}_{k\leq m}$, $\{\ov r(k,n)\}_{k\leq m}$, and collections $\C_\si$, where $\si$ ranges over the $m$-skeleton $P^{(m)}$.

First suppose that $m=0$, and pick a vertex $v\in P^{(0)}$.  Let $\C^0_v:=\{\ol{B_{\ov d_{p_v}}(x,2r)}\mid x\in \ov W_{p_v}\setminus W_{p_v}\}$, and hence $\{\Int Y\}_{Y\in \C^0_v}$ covers $\ov W_{p_v}\setminus Z^0_{p_v}$.  Applying Lemma~\ref{lem_preparatory_stuff}, we get a disjoint collection $\C^1_v$ of closed balls in $(\ov W_{p_v},\ov d_{p_v})$ such that $\{\Int Y\}_{Y\in\C^1_v}$ covers $\cup_{Y\in \C^0_v}Y$, where $|\C^1_v|\leq n$, and every  $Y\in \C^1_v$ has radius $<4^n\cdot 2r$. If  $4^n\cdot 2r<r_X$, then every $Y$ is a $3$-disk, and hence by Lemma~\ref{lem_preparatory_stuff} there is a unique  disjoint collection $\C_v$ of $3$-disks in $X$ such that $X\setminus \cup_{Y\in\C_v}Y=\ov W_{p_v}\setminus \cup_{Y\in \C^1_v}Y$.  Taking $c_0=1$, $c_0'=4^{n+2}\pi$, $\ov r(0)=c_0^{-1}r_X$, since $g(b_v)\leq 2g(p_v)$ on $Z^1_v=Z^0_{p_v}$ properties (a)--(d) follow immediately.  If $W_{p_v}=X$, then $\C^0_v=\emptyset$,   so   (e) holds.

Now suppose $m>0$, and assume that universal constants $\{c_k\}_{k\leq m-1}$, $\{c_k'(n)\}_{k\leq m-1}$, $\{\ov r(k,n)\}_{k\leq m-1}$, and collections $\C_\tau$ where $\tau$ ranges over the $(m-1)$-skeleton $P^{(m-1)}$, have been chosen so that (a)--(e) hold.

Pick an $m$-face $\si\subset P$.   

By our induction assumption, for every $\tau\subsetneq \si$, $\{\Int Y\}_{Y\in \C_\tau}$ covers $X\setminus Z^1_\tau$.  Since $Z^1_\tau\subset  Z^1_\si\subset W_{b_\si}$, it follows that $\{\Int Y\}_{Y\in\C_\tau}$  covers  $X\setminus Z^1_\si$, and we may apply Lemma~\ref{lem_preparatory_stuff} to obtain  a disjoint collection $\C^0_\tau$ of $3$-disks in $\ov W_{b_\si}$ such that $\ov W_{b_\si}\setminus\cup_{Y\in \C^0_\tau}Y=X\setminus \cup_{Y\in \C_\tau}Y$.  By (c) of our induction assumption and the fact that $2^{-1}g(b_\tau)< g(b_\si)< 2g(b_\tau)$ on $Z^1_\si$, for all $Y\in \C^0_\tau$ the boundary $2$-sphere $\D Y$ has intrinsic diameter $< 2c_{m-1}'(n)r$ with respect to $g(b_\si)$.  
So for all $Y\in \C^0_\tau$, the boundary $\D Y$ is contained in a $\ov d_{b_\si}$-ball $B_Y\subset \ov W_{b_\si}$ of radius $2c_{m-1}'(n)r$.  Therefore if $2c_{m-1}'(n)r<r_X$, we get that $\D Y$ bounds a $3$-disk $D_Y$ in $B_Y$.  We must have $D_Y=Y$, since $\D Y$ bounds two domains with boundary in $\ov W_{b_\si}$, precisely one of which is a $3$-disk; hence $Y\subset B_Y$.  
	
Now let $\C^0_\si=\cup\{\C^0_\tau\mid \text{$\tau\subset \si$ is an $(m-1)$-face}\}$.  Thus $|\C^0_\si|\leq (m+1)c_{m-1}n$.  Applying Lemma~\ref{lem_preparatory_stuff} to $\C^0_\si$, we obtain a disjoint collection $\C^1_\si$ of closed balls of radius $\leq 4^{(m+1)c_{m-1}n}\cdot 2c_{m-1}'(n)r$, where $|\C^1_\si|\leq (m+1)c_{m-1}n$, such that $\{\Int Y\}_{Y\in \C^1_\si}$ covers $\cup_{Y\in \C^0_\si}Y$.  Provided $r<4^{-(m+1)c_{m-1}n}(2c_{m-1}'(n))^{-1}r_X$, every $Y\in \C^1_\si$ will be a $3$-disk.
	
Applying Lemma~\ref{lem_preparatory_stuff}, there is a unique disjoint collection $\C_\si$ of $3$-disks in $X$ such that $X\setminus \cup_{Y\in \C_\si}Y=\ov W_{b_\si}\setminus\cup_{Y\in \C^1_\si}Y$.  
As $\{ \Int Y \}_{y \in \C^1_\sigma}$ covers $\C_\tau^1$ for every $(m-1)$-face $\tau$ of $\sigma$, we have 
\[ X \setminus \cup_{Y \in \C_\sigma} \Int Y = \ov{W}_{b_\sigma} \setminus \cup_{Y \in \C^1_\sigma} \Int Y \subset \ov{W}_{b_\sigma} \setminus \cup_{Y \in \C^0_\tau} \Int Y = X \setminus \cup_{Y \in \C_\tau} Y. \]
Therefore,  $\{\Int Y\}_{Y\in \C_\si}$ covers $\C_\tau$ for every $(m-1)$-face $\tau$ of $\si$, and hence by the induction assumption, it also covers $X\setminus Z^1_\si$.  Also, if    $W_{p_v}=X$ for every $v\in \si^{(0)}$, then by our induction assumption $\C_\tau=\emptyset$ for every $(m-1)$-face $\tau\subset\si$, which implies that $\C_\si=\emptyset$ as well.

Letting $c_m:=(m+1)c_{m-1}$, $c_m'(n)=4^{(m+1)c_{m-1}n}\cdot 2c_{m-1}'(n)\pi$, and $\ov r(k,n)=c_m'(n)^{-1}r_X$, assertions (a)--(e) of the claim follow.  
\end{proof}

\bigskip\bigskip
For every $\si\in \face P$ we let $Z_\si:=\ol{X\setminus\cup_{Y\in\C_\si}Y}$.  Hence assertion (a) of Lemma~\ref{lem_combinatorial_assignments} holds.  

If $\si_1,\si_2\in\face P$ and $\si_1\subsetneq\si_2$, then by assertion (b) of the claim $\{\Int Y\}_{Y\in\C_{\si_2}}$ covers $\cup_{Y\in\C_{\si_1}}Y$.  Hence
$$
Z_{\si_2}=X\setminus\cup_{Y\in\C_{\si_2}}\Int Y\subset X\setminus\cup_{Y\in\C_{\si_1}}Y=\Int Z_{\si_1}
$$
and assertion (b) of Lemma~\ref{lem_combinatorial_assignments} holds.

Pick $\si\in\face P$.  For every $\tau\subsetneq \si$, and every $p\in \si$, assertion (b) of the claim and the definition of $Z^1_\tau$ give
$$
Z_\si\subset\Int Z_\tau\subset Z^1_\tau\subset Z^1_\si\subset W_p\,.
$$
We let $U_\si$ be an open subset of $\cap_{\tau\subsetneq\si}\Int Z_\tau$ containing $Z_\si$ such that $U_\si\cap Y$ is simply-connected for all $Y\in\C_\si$.  If $\Psi_p:(W_p,g(p))\ra (X\setminus S_p,g_X)$ is an isometry as in assertion (ii) of Proposition~\ref{prop_extending_metrics_general}, then for every $Y\in \C_\si$, $p\in \si$, the composition 
$$
U_\si\cap Y\lra W_p\stackrel{\Psi_p}{\lra}X\setminus S_p\lra X
$$
is an isometric embedding of $(U_\si\cap Y,g(p))$ into $(X,g_X)$.  This lifts to an isometric embedding of $(U_\si\cap Y,g(p))$ into the universal cover of $(X,g_X)$, which is isometric to $(S^3,g_{S^3})$.  Thus assertions (c)(i) and (c)(ii) of Lemma~\ref{lem_combinatorial_assignments} hold.  If $\si\cap Q_0\neq\emptyset$, then for every $v\in \si^{(0)}$ we have $\ol{\St}(v,P)\cap Q_0\neq \emptyset$, which implies that $W_{p_v}=X$ by the choice of $p_v$.  Assertion (e) of the claim gives $\C_\si=\emptyset$, so $Z_\si=X$, and assertion (c)(iii) of Lemma~\ref{lem_combinatorial_assignments} holds.
\end{proof}

\bigskip\bigskip

\begin{proof}[Proof of Lemma~\ref{lem_extending_metrics_ball}]
The idea of the proof is to convert the extension problem for the family of $K\equiv 1$  metrics to an extension problem for embeddings into $S^3$, by working with (suitably normalized) isometric embeddings, rather than the metrics themselves.  The extension problem for the embeddings can then be solved by appealing to Hatcher's theorem.

Pick $x\in S^2$, $x'\in S^3$ and oriented  bases $e_1,e_2,e_3\subset T_{x}\R^3$,  $e_1',e_2',e_3'\subset T_{x'}S^3$ which are orthonormal with respect to $g_{\R^3}$ and $g_{S^3}$ respectively.  For each $p\in D^{m+1}$, apply the Gram-Schmidt process to $e_1,e_2,e_3$ to obtain an $h_{m+1}(p)$-orthonormal basis $f_1(p),f_2(p),f_3(p)$ for $T_{x}\R^3$.  For all $p\in D^{m+1}$ let $\psi_{m+1}(p):(N_{\rho}(S^2)\cap D^3,h_{m+1}(p))\ra (S^3,g_{S^3})$ be the isometric embedding that sends $f_1(p),f_2(p),f_3(p)$ to $e_1',e_2',e_3'$; similarly, for $p\in S^m$ let $\wh\psi_m(p):(D^3,\wh h_m(p))\ra (S^3,g_{S^3})$ be the isometric embedding sending $f_1(p),f_2(p),f_3(p)$ to $e_1',e_2',e_3'$.  By standard regularity for isometries, this yields continuous maps 
$$
\psi_{m+1}:D^{m+1}\ra \Embed_+(N_\rho(S^2)\cap D^3,S^3)\,,\qquad \wh\psi_m:S^m\ra \Embed_+(D^3,S^3)\,,
$$ 
where $\wh\psi_m(p)\restr (N_\rho(S^2)\cap D^3)=\psi_{m+1}(p)$ for all $p\in S^m$.

Next, we apply Lemma~\ref{lem_embedding_restriction_fiber_bundle} to produce  a continuous map $\wh\phi_{m+1}:D^{m+1} \linebreak[1] \ra \linebreak[1] \Embed_+(D^3, \linebreak[1] S^3)$ such that $\wh\phi_{m+1}(p)$ agrees with $\psi_{m+1}(p)$ on $S^2$ for all $p\in D^{m+1}$, and $\wh\phi_{m+1}(p)=\wh\psi_m(p)$ for all $p\in S^m$.  

Although $\wh\phi_{m+1}$ agrees with $\wh\phi_m$ on $S^2$, it may not agree with $\wh\phi_m$ near $S^2$.  To address this issue, we adjust $\wh\phi_{m+1}$ by precomposing it with a suitable family  of diffeomorphisms of $D^3$ that fix $\D D^3$ pointwise.   To that end, after shrinking $\rho$ if necessary, let $\wh\Phi_{m+1}:D^{m+1}\ra \Diff(D^3rel\D D^3)$ be a continuous map such that $\wh\phi_{m+1}(p)=\id_{D^3}$ for all $p\in S^m$, and $\wh\Phi_{m+1}(p)=\wh\phi_{m+1}^{-1}(p)\circ \psi_{m+1}(p)$ on $N_\rho(S^2)\cap D^3$ for all $p\in D^{m+1}$.  The map $\wh\Phi_{m+1}(p)$ may be obtained, for instance, by interpolating between $\wh\phi_{m+1}^{-1}(p)\circ \psi_{m+1}(p)$ and $\id_{D^3}$ using a partition of unity, i.e. letting 
$$
\wh\phi_{m+1}(p)(x)=u(|x|)\left(\wh\phi_{m+1}^{-1}(p)\circ \psi_{m+1}(p)(x)\right)+(1-u(|x|))x
$$
where $u:[0,1]\ra [0,1]$ is a smooth function supported in $[1-\eps,1]$ for some sufficiently small $\eps>0$, and $\eps\|\D_xu\|_{C^0},\eps^2\|\D^2_xu\|_{C^0}<C$ for some universal constant $C$. 

We now let $\wh\psi_{m+1}(p):=\wh\phi_{m+1}(p)\circ\wh\Phi_{m+1}(p)$ on $D^3$ for all $p\in D^{m+1}$. By the construction of $\wh\Phi_{m+1}$ it follows that $\wh\psi_{m+1}(p)$ agrees with $\psi_{m+1}(p)$ on $N_\rho(S^2)\cap D^3$, and $\wh\psi_{m+1}(p)=\wh\phi_{m+1}(p)$ for all $p\in S^m$.  Setting  $\wh h_{m+1}(p):=\wh\psi_{m+1}^*g_{S^3}$, we obtain the desired map $\wh h_{m+1}:D^{m+1}\ra \met_{K\equiv 1}(D^3,S^3)$.
\end{proof}

\section{Proof of Theorem~\ref{thm_k_equiv_1_contractible} for spherical space forms}
\label{sec_proof_of_spherical_space_form}

Let $X$ be a spherical space form other than $S^3$ or $RP^3$.  Choose $m\geq 0$, and a map  $h:S^m\ra\met_{K\equiv 1}(X)$.  

Let $g:D^{m+1}\ra\met(X)$ be a continuous extension of the composition $S^m\stackrel{h}{\ra}\met_{K\equiv 1}(X)\hookrightarrow\met(X)$.  For every $p\in D^{m+1}$, let $(W_{g(p)},\check g(p))\in\partmet_{K\equiv 1}(X)$ be the canonical partially defined metric constructed in Section~\ref{sec_canonical_limiting_metric}.  By Lemma~\ref{lem_continuous_partmet}, the assignment $p\mapsto (W_{g(p)},\check g(p))$ defines a continuous map $D^{m+1}\ra\partmet(X)$.  Taking $P_0=D^{m+1}$, $Q_0=S^m$,  by Lemma~\ref{lem_canonical_limiting_metric} the map $p\mapsto (W_{g(p)},\check g(p))$ satisfies the hypotheses of Proposition~\ref{prop_extending_metrics_general}.  Now the map $D^{m+1}\ni p\mapsto \wh g(p)$ furnished by  Proposition~\ref{prop_extending_metrics_general} defines a continuous extension $\wh g:D^{m+1}\ra \met_{K\equiv 1}(X)$ of $h$.

The argument above implies that the homotopy groups of $\met_{K\equiv 1}(X)$ are trivial.  The space $\met_{K\equiv 1}(X)$ is homotopy equivalent to a CW complex  (see Lemma~\ref{lem_structure_met_k}), so it is contractible.

\section{Proof of Theorem~\ref{thm_k_equiv_1_contractible} for hyperbolic manifolds}
The proof of Theorem~\ref{thm_k_equiv_1_contractible} for hyperbolic manifolds is the same as the proof for spherical space forms, apart from some minor changes, which we now explain.

In Section~\ref{sec_canonical_limiting_metric} (see Lemma~\ref{lem_continuous_partmet}) we used \cite{hamilton_positive_ricci,knopf_et_al} to show that when $Y$ is a spherical space form, if $g\in \met(Y)$ has positive sectional curvature, then modulo rescaling the maximal Ricci flow $(g(t))_{t\in[0,T)}$ with $g(0)=g$ converges in the $C^\infty$-topology to a metric $\ov g\in \met_{K\equiv 1}(Y)$ as $t$ tends to the blow-up time $T$, and the limit $\ov g$ depends continuously on $g$.  Here we may replace this with the assertion that for every $g_X\in \met_{K\equiv -1}(X)$, there is an $\eps_{g_X}>0$  such that  if $g\in \met(X)$ is $\eps_{g_X}$-close to $g_X$ and $(g(t))_{t\in [0,T)}$ is the maximal Ricci flow with $g(0)=g$, then:
\begin{enumerate}[label=(\roman*)]
\item $T=\infty$.
\item Modulo rescaling $g(t)$ converges in the $C^\infty$-topology to a metric $\ov g\in \met_{K\equiv -1}(X)$ as $t\ra \infty$.
\item The limit metric $\ov g$ depends continuously on $g$.  
\een
Statements (i) and (ii) follow immediately from the convergence of normalized Ricci flow shown in \cite{ye_convergence,bamler_stability_hyperbolic_cusps}.
The continuity assertion (iii) is a consequence of the uniform exponential decay of the time derivatives, which follows readily from their arguments. 

We now adapt the results in Section~\ref{sec_canonical_limiting_metric} to the hyperbolic case.  The statements are nearly identical, apart from  obvious changes.

Let $X$ be a compact, connected,  hyperbolic manifold.  We first assume in addition that $X$ is orientable; we will remove this assumption below.

Pick a hyperbolic metric $g_X\in \met_{K\equiv -1}(X)$.  Pick $g\in \met(X)$ and let $\M$ be a singular Ricci flow with $\M_0=(X,g)$.  By Theorem~\ref{thm_structure_singular_ricci_flow},  for every $t<\infty$ there is a unique component $C_t$ of $\M_t$ that is a punctured copy of $X$.  

For every $t_1,t_2\in[0,\infty)$, let $C_{t_1,t_2}\subset C_{t_1}$ be the set of points in $C_{t_1}$ that survive until time $t_2$, i.e. the points for which the time $(t_2-t_1)$-flow of the time vector field $\D_\t$ is defined.  Then $C_{t_1,t_2}$ is an open subset of $C_{t_1}$, and the time $(t_2-t_1)$-flow of $\D_\t$ defines a smooth map $\Phi_{t_1,t_2}:C_{t_1,t_2}\ra \M_{t_2}$, which is a diffeomorphism onto its image.  We define a metric $g_{t_1,t_2}$ on $\Phi_{t_1,t_2}(C_{t_1,t_2})$ by $g_{t_1,t_2}:=(\Phi_{t_1,t_2})_*g_{t_1}$, where $g_{t_1}$ is the spacetime metric on $C_{t_1,t_2}\subset \M_{t_1}$.  We let $W_g(t):=\Phi_{t,0}(C_{t,0})$, so  $g_{t,0}$ is a metric on $W_g(t)$.

\begin{lemma}[Limiting $K\equiv -1$ metric]
\label{lem_canonical_limiting_k_equiv_minus_1_metric}
Choose $t_0<\infty$ such that $C_{t_0}$ is compact and $\eps_{g_X}$-close to $(X,g_X)$, where $\eps_{g_X}>0$ is the constant for which (i)--(iii) hold.  Then:  
\ben
\item $W_g(t)=W_g(t_0)=:W_g$ for all $t\in[t_0,\infty)$.
\item Modulo rescaling, $g_{t,0}$ converges in the smooth topology to a $K\equiv -1$ metric $\check g$ on $W_g$ as $t\ra \infty$.
\item $(W_g,\check g)$ is isometric to $(X\setminus S, g_X)$ for some finite (possibly empty) subset $S\subset X$, where the cardinality of $S$ is bounded above depending only on $t_0$, and the bounds the curvature, injectivity radius, and volume of $g$.
\item If $g$ has constant sectional curvature, then $W_g=X$ and $\check g=\lambda g$ for some $\lambda\in(0,\infty)$.
\een
\end{lemma}
The proof is nearly identical to the proof of Lemma~\ref{lem_canonical_limiting_metric}, except that we appeal to Theorem~\ref{thm_structure_singular_ricci_flow}(5) rather than Theorem~\ref{thm_structure_singular_ricci_flow}(4).

As in the spherical space form case, Theorem~\ref{thm_structure_singular_ricci_flow} implies that $(W_g,\check g)$ is well-defined.  Then the proof of continuity in Lemma~\ref{lem_continuous_partmet} carries over, using Theorem~\ref{thm_structure_singular_ricci_flow}(5) and (iii) instead of Theorem~\ref{thm_structure_singular_ricci_flow}(4) and \cite{knopf_et_al}.

Now assume that $X$ is not orientable.

Let $\wh X\ra X$ be the $2$-fold orientation cover, with deck group action $\Z_2\acts\wh X$.  For every $g\in \met(X)$, let $\wh g,\wh g_X\in \met(\wh X)$ be the pullbacks of $g$, $\wh g$ to $\wh X$, respectively, and let  $\wh \M$ be a singular Ricci flow with $\M_0$ given by $(\wh X,\wh g)$.  

By Theorem~\ref{thm_existence_uniqueness_singular_ricci_flow}, the deck group action $\Z_2\acts \wh X=\M_0$ extends uniquely to an isometric action $\Z_2\acts\wh\M$.  Since the partially defined metric $(W_{\wh g},\wh {\check g})\in\partmet_{K\equiv -1}(\wh X)$ is canonical and depends continuously on $g$, it is $\Z_2$-invariant, and descends to $X$, yielding a partially defined metric $(W_g, \check g)\in\partmet_{K\equiv -1}(X)$, which depends continuously on $g$.  

By Lemma~\ref{lem_canonical_limiting_k_equiv_minus_1_metric} the Riemannian manifold $(W_{\wh g},\wh g_{\om})$ is isometric to $(\wh X,\wh g_X)$ punctured at a finite set of points of cardinality controlled by bounds on the geometry of $g$ and $t_0$, as in assertion (3) of Lemma~\ref{lem_canonical_limiting_k_equiv_minus_1_metric}.
Therefore the metric completion $\ol{(W_{\wh g},\wh {\check g})}$ of $(W_{\wh g},\wh {\check g})$  is isometric to $(\wh X,\wh g_X)$.   The isometric action $\Z_2\acts (W_{\wh g},\wh {\check g})$ extends canonically to an isometric action $\Z_2\acts \ol{(W_{\wh g},\wh {\check g})}$.  Suppose some point $x\in \ol{(W_{\wh g},\wh {\check g})}\setminus(W_{\wh g},\wh {\check g})$ is fixed by the $\Z_2$ action.  Then the $\Z_2$-action will preserve a small metric sphere $N\subset W_{\wh g}$ centered at $x$, preserving an orientation of its normal bundle.  It follows that the quotient $N/\Z_2$ is a $2$-sided copy of $RP^2$ embedded in $X$.  This contradicts the fact that $X$ is a compact hyperbolic manifold. Therefore the action $\Z_2\acts \ol{(W_{\wh g},\wh {\check g})}$ is free.  Hence the completion of $(W_g,\check g)$ is a compact hyperbolic $3$-manifold $(X',g'_{X'})$, and $(W_g,\check g)$ is isometric to $(X'\setminus S',g'_{X'})$.   Now the embedding $X'\setminus S'\simeq W_g\hookrightarrow X$, together with the irreducibility of $X$, implies that $X'$ is diffeomorphic to $X$.

Thus we have shown that we have a well-defined continuous assignment $g\mapsto (W_g,\check g)\in\partmet(X)$ such that $(W_g,\check g)$ is isometric to $(X\setminus S_g,g_X)$ for some finite set $S_g\subset X$, where the cardinality $|S_g|$ is controlled by the constant $t_0=t_0(g)$ as in Lemma~\ref{lem_canonical_limiting_k_equiv_minus_1_metric}, and bounds on the geometry of $g$.  However, by Theorem~\ref{thm_convergence_singular_ricci_flows} and the convergence property stated in (i)--(iii) in the beginning of this section we may choose $t_0(g)$ to be a locally bounded function of $g$.  In particular, on any compact subset of $\met(X)$, we may choose $g\mapsto t_0(g)$ to be bounded.

Sections~\ref{sec_extending_constant_curvature_metrics} and \ref{sec_proof_of_spherical_space_form} now carry over after making the obvious change in the sign of curvature, and replacing $S^3$ with hyperbolic $3$-space $\H^3$.  

\bigskip\bigskip

\begin{remark}
(cf. Remark~\ref{rem_alternate_continuous_dependence}) Rather than using the uniform exponential convergence of the normalized Ricci flow to a hyperbolic metric, for the applications in this paper it would work equally well if we used another procedure for converting an almost  hyperbolic metric to a hyperbolic metric.  For instance, if $(X,g_X)$ is a hyperbolic manifold, and $g\in \met(X)$ is sufficiently close to $g_X$, then there is a unique harmonic map $\phi_g:(X,g)\ra (X,g_X)$ homotopic to the identity map $\id_X$, and the pullback $\phi_g^*g_X$ is a hyperbolic metric.  Furthermore (by a simple compactness argument) $\phi_g^*g_X$ varies continuously with $g$.
\end{remark}

\bibliography{gsc}{}
\bibliographystyle{amsalpha}

\end{document}